 \documentclass[a4paper,11pt]{article}
\usepackage{pdfsync}

 \usepackage[plainpages=false]{hyperref}
 \usepackage{amsfonts,amsmath,amssymb,amsthm}
 \usepackage{latexsym,lscape,rawfonts,mathrsfs}

 \usepackage[dvips]{color}
 \usepackage{multicol}

 \usepackage[all]{xy}
 \usepackage{eufrak}
 \usepackage{makeidx}
 \usepackage{graphicx,psfrag}
 \usepackage{pstool}

 \usepackage{array,tabularx}

 \usepackage{setspace}

 \usepackage{appendix}

\usepackage{txfonts}

 \newcommand{\R}{\ensuremath{\mathbb{R}}}

 \newcommand{\ba}{\begin{align*}}
 \newcommand{\ea}{\end{align*}}

 \DeclareMathOperator{\Vol}{Vol}

 \makeatletter
 \def\ExtendSymbol#1#2#3#4#5{\ext@arrow 0099{\arrowfill@#1#2#3}{#4}{#5}}
 
 \makeatother

 \makeatletter
 \def\ExtendSymbol#1#2#3#4#5{\ext@arrow 0099{\arrowfill@#1#2#3}{#4}{#5}}
 \newcommand\longright[2][]{\ExtendSymbol{-}{-}{\rightarrow}{#1}{#2}}
 \makeatother

 \definecolor{hao}{rgb}{1,0.5,0}
 \definecolor{miao}{cmyk}{0.5,0,0.2,0.2}
 \definecolor{qiao}{gray}{0.96}

\newtheorem{prop}{Proposition}[section]

\newtheorem{proposition}[prop]{Proposition}

\newtheorem{theorem}[prop]{Theorem}

\newtheorem{lemma}[prop]{Lemma}

\newtheorem{corollary}[prop]{Corollary}

\newtheorem{remark}[prop]{Remark}

\newtheorem{definition}[prop]{Definition}

\newtheorem{conjecture}[prop]{Conjecture}

\numberwithin{equation}{section}

 \title{Remarks of weak-compactness along K\"ahler Ricci flow}
 
 \author{Xiuxiong Chen\footnote{Supported by NSF grant DMS-1515795.}\;,  Bing Wang\footnote{Supported by NSF grant DMS-1510401.}}
 
 \date{}

\begin{document}
  \maketitle


\section{Introduction}

The purpose of this note is neighther originality nor generality, but to give a direct proof of the following weak compactness theorem(Theorem~\ref{thm:MD09_1}), using the methods developed in our previous work Chen-Wang~\cite{CW6}.  
Since the scope of this paper is limited and focused,  we are able to make the argument more streamlined and accessible.
For the readers who are interested in the original motivations, global pictures
as well as historical perspective of this problem, we refer them to Chen-Wang~\cite{CW6} for full details. 



\begin{theorem}[\textbf{Special case of Chen-Wang~\cite{CW6}}]
    Suppose $\{(M^n, g(t), J), 0 \leq t< \infty\}$ is a K\"ahler Ricci flow solution 
    \begin{align}
       \frac{\partial}{\partial t} g_{ij}= -R_{ij} + g_{ij},   \quad  \label{eqn:MD16_3}
    \end{align}
    in the class $2\pi c_1(M, J)$, where $(M, J)$ is a Fano manifold of complex dimension $n$. 
    For each $t_i \to \infty$, by taking subsequence if necessary, we have $(M, g(t_i))$ converges to a limit space $(\hat{M}, \hat{g})$ in the Gromov-Hausdorff topology. 
    The limit space space $\hat{M}$ has a regular-singular decomposition $\hat{M}=\mathcal{R} \cup \mathcal{S}$, where $\mathcal{R}$ and $\mathcal{S}$ satisfies the following properties.
    \begin{itemize}
    \item $\mathcal{R}$ is an open, smooth manifold with a complex structure $\hat{J}$ and a smooth holomorphic vector field $\nabla \hat{f}$ such that 
       \begin{align}
           Ric(\hat{g})-\hat{g}= -L_{\nabla \hat{f}} (\hat{g}).  \label{eqn:MD09_1}
       \end{align}   
       In other words, $(\mathcal{R}, \hat{g}, \hat{J}, \hat{f})$ is a K\"ahler Ricci soliton. 
    \item The Hausdorff dimension of $\mathcal{S}$ is less or equal to $2n-4$. 
    \end{itemize}
 Furthermore, the convergence topology can be improved to be the Cheeger-Gromov topology.   
 \label{thm:MD09_1}   
 \end{theorem}
 
 Theorem~\ref{thm:MD09_1} answers a long-standing conjecture of K\"ahler Ricci flow on Fano manifolds.   More information of the background of this theorem can be found in Chen-Wang~\cite{CW6}. 
 For the convenience of the readers, we copy down the first written down statement of the conjecture related to Theorem~\ref{thm:MD09_1} as follows. 
  
\begin{conjecture}[Tian~\cite{Tian97}]
By taking subsequences if necessary, one should have that $(M, \omega_t)$ converges to a space $(M_{\infty}, \omega_{\infty})$, 
 which is smooth outside a subset of real Hausdorff codimension at least $4$, in the Cheeger-Gromov-Hausdorff topology. Furthermore, $(M_{\infty}, \omega_{\infty})$ can be expanded to be an obstruction triple $(M_{\infty}, v, \xi)$(possibly singular)
 satisfying:
 \begin{align*}
   Ric(\omega_{\infty}) -\omega_{\infty}=-L_{v}(\omega_{\infty}),  \quad \textrm{on the regular part of} \; M_{\infty}, 
 \end{align*}
 where $L_{v}$ denotes the Lie derivative in the direction of $v$. In particular, $(M_{\infty}, \omega_{\infty})$ is a Ricci soliton if $\omega_{\infty}$ is not K\"ahler-Einstein. 
\label{cje:MD03_1} 
\end{conjecture}

The above statement follows the exact words of Conjecture 9.1 of Tian~\cite{Tian97}.  In the same paper, concerning the convergence of the K\"ahler Ricci flow, Tian wrote down the following words on page 36. 
 
 \textit{Previously, R. Hamilton thought that the limit $(M_{\infty}, \omega_{\infty})$ should be a Ricci soliton.  Our new observation here is that $\omega_{\infty}$ may be K\"ahler-Einstein, and otherwise, it is a special Ricci soliton.}
 
 \noindent
 Note that K\"ahler-Einstein metrics are trivial Ricci soliton metrics.  The convergence part of Conjecture~\ref{cje:MD03_1} was solved in Chen-Wang~\cite{CW6}, as a special case of weak compactness theory of polarized K\"ahler Ricci flows. 
 Recall that the objects of the weak-compactness in Chen-Wang~\cite{CW6} are the ``flows", not only ``time slices".  
 For the purpose of flow compactness and to show that the limit $M_{\infty}$ is a projective variety, one cannot avoid the application of Bergman kernel, as done in Donadson-Sun~\cite{DS}, at least in the current stage.
 The extra structures, i.e., the flow structure, the line bundle structure and the variety structure, of the limit are of essential importance for further applications, like the flow proof of Yau's stability conjecture. 
 This point was  emphasized on page 2 of  Chen-Sun-Wang~\cite{CSW}. 
 However, if one only want to prove the convergence of the time slices of the  K\"ahler Ricci flow, i.e., the convergence part of Conjecture~\ref{cje:MD03_1}, or Theorem~\ref{thm:MD09_1}, then a much simpler proof can be given.
 This is the purpose of this paper.    Since Theorem~\ref{thm:MD09_1} does not involve the variety structure and line bundle structure,  its requirement of K\"aher geometry is very limited(c.f. Remark~\ref{rmk:MD02_1}).
 No new idea beyond Chen-Wang~\cite{CW6} is needed. \\

 Theorem~\ref{thm:MD09_1} is a ``regularity improvement" theorem. 
 In fact,  because of the breakthrough of Perelman~\cite{Pe1}, along the K\"ahler Ricci flow, it is well known that diameter, scalar curvature and non-collapsing constant are all uniformly bounded(c.f. Perelman's results written down by Sesum-Tian~\cite{SeT}).
 On the other hand, based on the work of Chen-Wang~\cite{CW5} and Q.S. Zhang~\cite{Zhq3}, we have uniform non-inflating condition.  
 Namely, for each $r \in (0, 1)$ and $(x,t) \in M \times [0, \infty)$,  there is a uniform $\kappa$ independent of $r$ and $(x,t)$ such that
 \begin{align}
   \kappa \leq   \omega_{2n}^{-1}r^{-2n} |B(x, r)|_{g(t)} \leq \kappa^{-1}       \label{eqn:MD10_1}
 \end{align}
 along the flow, where $\omega_{2n}$ is the volume of unit ball in $\R^{2n}$. 
 Since K\"ahler Ricci flow in $2\pi c_1(M, J)$ preserves volume, it follows from standard ball packing argument that $(M, g(t_i))$ converges(by passing to subsequence if necessary) to a limit compact length space
 $(\hat{M}, \hat{g})$ in Gromov-Hausdorff topology. 
 Theorem~\ref{thm:MD09_1} basically says that both $\hat{M}$ and the convergence topology to $\hat{M}$ have good regularity. 
 Therefore, it is a ``regularity improvement" theorem in nature. \\

 The proof of  Theorem~\ref{thm:MD09_1} consists of two basic steps:
   
 \textit{Step 1. Develop the rough weak compactness under the canonical radius assumption.}
 
 \textit{Step 2. Using the intrinsic Ricci flow structure to obtain precise weak compactness and obtain the a priori estimate of canonical radius.}
 
 The  ``canonical radius" (with respect to some singular model space) is a crucial new ingredient in Chen-Wang~\cite{CW6}. 
 Roughly speaking, we need to first choose a model space with compact moduli under proper topology. 
 After the choice of model space,  the canonical radius can be regarded as the largest scale such that the manifold can be well-approximated by model space. 
 For example, one can regard the harmonic radius(c.f. Anderson~\cite{An90}) as the canonical radius with respect to the model space $\R^{m}$. 
 In Perelman's fundamental work~\cite{Pe1},  the model space of 3-dimensional Ricci flows is the $\kappa$-solution, i.e., the $\kappa$-noncollapsed, ancient Ricci flow solution with nonnegative sectional curvature(c.f. section 11 of Perelman~\cite{Pe1}).
 Note that 3-dimensional $\kappa$-solution's moduli space is compact  under the smooth topology.
 Using $\kappa$-solution as model space, one can define canonical radius of a space-time point with respect to $\kappa$-solution and smooth topology.   
 Then the ``canonical neighborhood" theorem, i.e.,  Theorem 12.1 of Perelman~\cite{Pe1},  can be understood as the canonical radius of a high curvature point in the 3-dimensional Ricci flow is uniformly bounded from below.
 It is this ``canonical neighborhood" theorem that motivates us to use the term ``canonical radius".  
 
 However, before Chen-Wang~\cite{CW3}, \cite{CW6}, all the model spaces and related topology are smooth. 
 In the first paper of Chen-Wang~\cite{CW3}, we used the K\"ahler-Ricci-flat surface orbifolds as model space and pointed-Cheeger-Gromov topology as the proper topology to study the convergence of K\"ahler Ricci flow on Fano surfaces. 
 The compactness of the moduli of noncollapsed K\"ahler-Ricci-flat surface orbifolds(c.f. Anderson~\cite{An05}) plays an important role in Chen-Wang~\cite{CW3}. 
 In the second paper of Chen-Wang~\cite{CW6}, the name of ``canonical radius" was written down explicitly.   
 An essential step beyond the first paper is to figure out the exact model space, which does not exist in literature before the second paper Chen-Wang~\cite{CW6}.  The discovery of the model space was lead by the following speculation:
  
  \textit{It should be very hard to make difference between the blowup limits from K\"ahler Einstein metrics and the K\"ahler Ricci flows with bounded scalar curvature, since both of them should be scalar flat and consequently Ricci-flat.}
  
 \noindent
 Such motivation was already explained in the first paper of Chen-Wang~\cite{CW3}, at the end of section 2. 
 From the above speculation,  the model space should have all the properties that a K\"ahler Einstein limit space have. However,  it should be slightly bigger than the space of K\"ahler Einstein limit spaces, since the space of  K\"ahler Ricci flows
  is bigger than the space of  K\"ahler Einstein manfiolds.  
  Along this route, we finally found that the natural model space should be the non-collapsed K\"ahler Ricci flat manifolds with mild singularities, 
  whose precise definition can be found in Definition~\ref{dfn:GC21_1}. 
  The relationship between the model space and K\"ahler Einstein blowup limit space and K\"ahler Einstein manifolds are illustrated in Figure~\ref{fig:threemoduli}. 
  Among the defining properties of the model space,  the following three are crucial:
  
  \begin{itemize}
  \item Gap between regular part $\mathcal{R}$ and singular part $\mathcal{S}$.
  \item High codimension($>4-\epsilon$) of singular part $\mathcal{S}$. 
  \item Convexity of regular part $\mathcal{R}$. 
  \end{itemize}
  
  \noindent
  Each of the above listed property corresponds to an estimate of spaces very close (in the Cheeger-Gromov topology) to the model space(c.f. Figure~\ref{fig:propertyestimate}). 
  Such estimates are used to define the canonical radius(c.f. Definition~\ref{dfn:SC02_1}).  
  Not surprisingly, the canonical radius of a point in the model space should be infinite(c.f. Remark~\ref{rmk:MD24_1}). 
  If canonical radius is uniformly bounded from below by $1$ for a space sequence, then the sequence(up to taking subsequences) has a weak limit $\bar{M}$ in the Cheeger-Gromov topology, as application of estimates on
  the right side of Figure~\ref{fig:propertyestimate}.   This is the key of the first step of the proof of Theorem~\ref{thm:MD09_1}.

\begin{figure}
 \begin{center}
 \psfrag{A}[c][c]{\color{red} Geometry on model space}
 \psfrag{B}[c][c]{\color{blue} Analysis on approximating manifolds}
 \psfrag{C}[c][c]{\color{red} Gap between $\mathcal{R}$ and $\mathcal{S}$}
 \psfrag{D}[c][c]{\color{red} High codimension of $\mathcal{S}$}
 \psfrag{E}[c][c]{\color{red} Convexity of $\mathcal{R}$}
 \psfrag{F}[c][c]{\color{blue} Regularity estimate}
 \psfrag{G}[c][c]{\color{blue} Density estimate}
 \psfrag{H}[c][c]{\color{blue} Connectivity estimate}
 \includegraphics[width=0.8 \columnwidth]{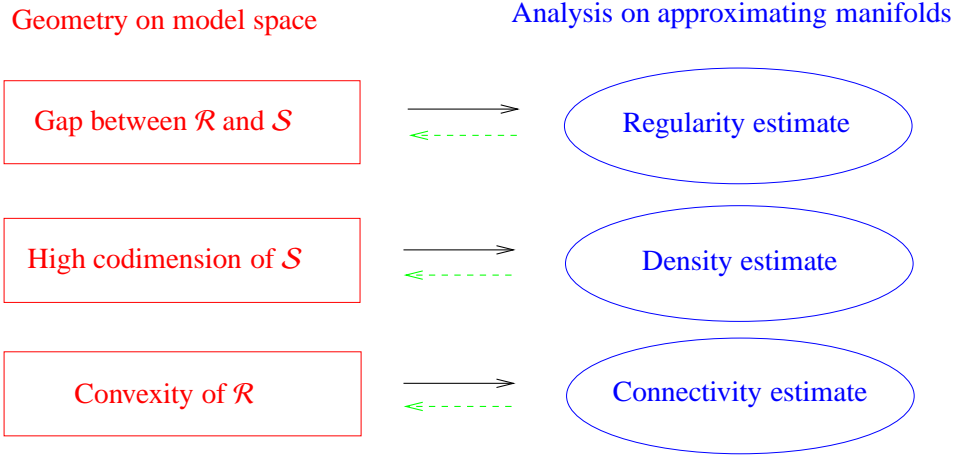}
 \caption{Geometry property and analysis estimate}
 \label{fig:propertyestimate}
 \end{center}
 \end{figure}

  Then we discuss the second step of the proof of Theorem~\ref{thm:MD09_1}, i.e., developing the uniform lower bound of canonical radius(or its generalized version) along the K\"ahler Ricci flow.  
  This bound is carried out by a contradiction argument.   
  For otherwise, one can find a sequence whose canonical radius is tending to zero. 
  Applying a point-selecting technique and rescaling argument, we obtain a sequence of Ricci flows whose canonical radii at base points are all  $2$ and nearby canonical radii are at least $1$.
  Moreover, scalar curvature tends to $0$. 
  Let $(M_i, x_i, g_i(0))$ be central time slices of such a sequence.  It has a limit space $(\bar{M}, \bar{x}, \bar{g})$ by the uniform lower bound of canonical radius.  
  Using canonical radius assumption, $\bar{M}$ has only weak version of the geometry property on the left side of Figure~\ref{fig:propertyestimate}.
  For example, by connectivity estimate, we only know that the regular part of $\bar{M}$ is $3$-connected(c.f. Proposition~\ref{prn:MD16_1}).
  In other words, every two points  $x, y \in \mathcal{R}(\bar{M})$ can be connected by a smooth curve in $\mathcal{R}$ whose length is less than $3d(x,y)$.  
  In order $\bar{M}$ to locate in the model space, we need the length of the curve to be exactly $d(x,y)$. 
  The intrinsic geometry of the Ricci flows,  e.g.,  the reduced geodesic, reduced distance, and reduced volume  of Perelman~\cite{Pe1} are essentially used to improve the regularity of $\bar{M}$.
  Applying such intrinsic geometry,  all the geometric properties on the left side of  Figure~\ref{fig:propertyestimate} hold. 
  Therefore, $\bar{M}$ locates in the model space and $\bar{x}$ has infinity canonical radius.   This is a contradiction by the weak continuity of canonical radius under the Cheeger-Gromov topology(c.f. Proposition~\ref{prn:MD03_2}). 
  Technically, in order to apply the intrinsic geometry of the Ricci flow, we need metric distortion estimate(c.f. Lemma~\ref{lma:MA13_2}, Remark~\ref{rmk:MD28_3}) 
  and weak long-time two-sided pseudo-locality(c.f. Proposition~\ref{prn:GC21_10}, Remark~\ref{rmk:MD28_3}). 
  Note that  the K\"ahler condition is only  used in the compactness of  the moduli of model space.
  However, with Cheeger-Naber~\cite{CN2}, the method of Chen-Wang~\cite{CW6} can be generalized over to the Riemannian case seamlessly(c.f. Remark~\ref{rmk:MD09_1}, 
  Remark~\ref{rmk:GC23_2} and Remark~\ref{rmk:MD02_1}),
  for the purpose of studying the time slice weak compactness of the non-collapsed Ricci flows with bounded scalar curvature.
  Largely following the framework of Chen-Wang~\cite{CW6},  interesting progress in this generalization was considered in a recent work of Bamler~\cite{Bamler}. \\

  The structure of this paper is as follows.  In section 2, we describe the precise definition of model space and canonical radius, together with some other auxiliary definitions.
  In section 3, we discuss the principle to obtain metric distorsion, which originates from section 5.3 of Chen-Wang~\cite{CW6}. 
  In section 4 and section 5, we write down the details of Step 1 and Step 2 described above.   Finally, we finish the proof of Theorem~\ref{thm:MD09_1} in section 6.

\section{Model space and canonical radii}

Let $\mathscr{KS}(n,\kappa)$ be the moduli of complex $n$-dimensional K\"ahler Ricci-flat manifolds with asymptotic volume ratio at least $\kappa$.
Clearly, $\mathscr{KS}(n,\kappa)$  is not compact under the pointed-Gromov-Hausdorff topology.
It can be compactified as a space $\overline{\mathscr{KS}}(n,\kappa)$.
 However, this may not be the largest space that one can develop weak-compactness theory. 
So we extend the space $\overline{\mathscr{KS}}(n,\kappa)$ further to a possibly bigger compact space $\widetilde{\mathscr{KS}}(n,\kappa)$, which is defined as follows.
Note that for convenience of notations, we denote $m=2n$.   In short,  \textbf{$m$ is the real dimension, $n$ is the complex dimension.}

\begin{figure}
 \begin{center}
 \psfrag{A}[c][c]{$\color{green}\mathscr{KS}(n,\kappa)$}
 \psfrag{B}[c][c]{$\color{blue}\overline{\mathscr{KS}}(n,\kappa)$}
 \psfrag{C}[c][c]{$\color{red} \widetilde{\mathscr{KS}}(n,\kappa)$}
 \includegraphics[width=0.5 \columnwidth]{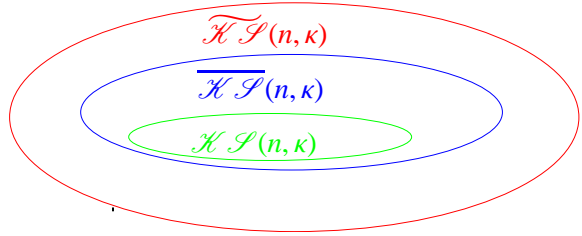}
 \caption{Relations among different moduli spaces}
 \label{fig:threemoduli}
 \end{center}
 \end{figure}

 \begin{definition}(c.f. Definition 2.1. of Chen-Wang~\cite{CW6})
 Let $\widetilde{\mathscr{KS}}(n,\kappa)$ be the collection of length spaces $(X,g)$ with the following properties.
\begin{enumerate}
  \item  $X$ has a disjoint regular-singular decomposition $X=\mathcal{R} \cup \mathcal{S}$, where $\mathcal{R}$ is the regular part,  $\mathcal{S}$ is the singular part.
   A point is called regular if it has a neighborhood which is isometric to a totally geodesic convex domain of  some smooth Riemannian manifold.  A point is called singular
   if it is not regular.
  \item  The regular part $\mathcal{R}$ is a nonempty, open K\"ahler Ricci-flat manifold of complex dimension $n$, real dimension $m$.
  \item  $\mathcal{R}$ is weakly convex, i.e., for every point $x \in \mathcal{R}$, there exists a measure zero set
    $\mathcal{C}_x$ such that every point in $X \backslash \mathcal{C}_x$ can be connected to $x$
    by a unique shortest geodesic in $\mathcal{R}$.   For convenience, we call $\mathcal{C}_x$ as the cut locus of $x$.
  \item $\dim_{\mathcal{M}} \mathcal{S} < m-3$, where $\mathcal{M}$ means Minkowski dimension.
  \item Let $\mathrm{v}$ be the volume density function, i.e.,
        \begin{align}
          \mathrm{v}(x) \triangleq \lim_{r \to 0} \frac{|B(x,r)|}{\omega_{m} r^{m}}   \label{eqn:SC16_1}
        \end{align}
        for every $x \in X$. Then $\mathrm{v}\equiv 1$ on $\mathcal{R}$ and $\mathrm{v} \leq 1-2\delta_0$ on $\mathcal{S}$.
        In other words, the function $\mathrm{v}$ is a criterion function for singularity.
        Here $\delta_0$ is the Anderson constant.
  \item The asymptotic volume ratio $\mathrm{avr}(X) \geq \kappa$. In other words, we have
   \begin{align*}
      \lim_{r \to \infty} \frac{|B(x,r)|}{\omega_{m}r^{m}} \geq \kappa
   \end{align*}
    for every $x \in X$.
\end{enumerate}
\label{dfn:GC21_1}
\end{definition}

The relationships among the three moduli spaces $\mathscr{KS}(n,\kappa)$, $\overline{\mathscr{KS}}(n,\kappa)$ and $\widetilde{\mathscr{KS}}(n,\kappa)$ can be seen from Figure~\ref{fig:threemoduli}. 
Note that $\mathscr{KS}(n,\kappa)$ is the objects of study in the  traditional Cheeger-Colding theory. 
By taking completion, many estimates in $\mathscr{KS}(n,\kappa)$ can be carried over to $\overline{\mathscr{KS}}(n,\kappa)$. 
A key observation of Chen-Wang~\cite{CW6} is that although $\widetilde{\mathscr{KS}}(n,\kappa)$ is not the completion of $\mathscr{KS}(n,\kappa)$, we can still obtain essential a priori estimates for
$\widetilde{\mathscr{KS}}(n,\kappa)$, because the singularities of each space in $\widetilde{\mathscr{KS}}(n,\kappa)$ are very mild.
The methods of Cheeger-Colding can be applied on $\widetilde{\mathscr{KS}}(n,\kappa)$ following mostly the route of the original Cheeger-Colding theory. 
The property 3, 4 and 5 in Definition~\ref{dfn:GC21_1} are the essential reasons why the singularities are mild. 
Furthermore, by extending the study objects from $\mathscr{KS}(n,\kappa)$ to $\widetilde{\mathscr{KS}}(n,\kappa)$,  many theorems can be stated more precisely, as done in section 2 of Chen-Wang~\cite{CW6}.   
In particular, we have the following theorems.

\begin{theorem}(c.f. Theorem 1.1. of Chen-Wang~\cite{CW6})
 $\widetilde{\mathscr{KS}}(n,\kappa)$ is compact under the pointed Cheeger-Gromov topology.
Moreover, each space $X \in \widetilde{\mathscr{KS}}(n,\kappa)$ is a K\"ahler Ricci-flat conifold in the sense of Chen-Wang~\cite{CW6}. 
\label{thm:GC21_2}
\end{theorem}

Basically, Theorem~\ref{thm:GC21_2} says that $X \in \widetilde{\mathscr{KS}}(n,\kappa)$ has better properties that it seems to be. 
For example, following definitions, one only knows that $\dim_{\mathcal{M}} \mathcal{S}<2n-3$ and $\mathcal{R}$ is weakly convex.
However, by delicate analysis based on all the defining properties, one can obtain that $\dim_{\mathcal{M}} \mathcal{S}<2n-4-\epsilon$ for each $\epsilon>0$, and $\mathcal{R}$ is strongly convex. 
These can be regarded as ``regularity improvement". 
There are many other ``regularity improvement" for spaces in $\widetilde{\mathscr{KS}}(n,\kappa)$. More details can be found in  Chen-Wang~\cite{CW6}.

For the purpose of ``almost" regular-singular decomposition of a manifold very close to the model space, we give the following definition. 

\begin{definition}(c.f. Definition 3.3. of Chen-Wang~\cite{CW6})
Denote the set 
\begin{align*}
\left\{ r \left| 0<r<\rho, \omega_{m}^{-1}r^{-m}|B(x_0,r)|\geq 1-\delta_0 \right. \right\}
\end{align*}
by $I_{x_0}^{(\rho)}$ where $x_0 \in M$, $\rho$ is a positive number.
Clearly, $I_{x_0}^{(\rho)} \neq \emptyset$ since $M$ is smooth. Define
  \begin{align*}
    \mathbf{vr}^{(\rho)}(x_0) \triangleq  \sup I_{x_0}^{(\rho)}.
  \end{align*}
For each pair $0<r \leq \rho$, define
\begin{align*}
  &\mathcal{F}_{r}^{(\rho)}(M) \triangleq \left\{ x \in M | \mathbf{vr}^{(\rho)}(x) \geq r  \right\}, \\
  &\mathcal{D}_{r}^{(\rho)}(M) \triangleq \left\{ x \in M | \mathbf{vr}^{(\rho)}(x) < r  \right\}.
\end{align*}
\label{dfn:SC24_1}
\end{definition}

\begin{definition}(c.f. Definition 3.5. of Chen-Wang~\cite{CW6})
We say that the canonical radius (with respect to model space $\widetilde{\mathscr{KS}}(n,\kappa)$, $m=2n$) of a point $x_0 \in M$ is
not less than $r_0$ if for every  $r < r_0$, we have the following properties.
\begin{enumerate}
  \item Volume ratio estimate: $\kappa \leq \omega_{m}^{-1}r^{-m}|B(x_0,r)| \leq \kappa^{-1}$.
  \item Regularity estimate: $r^{2+k}|\nabla^k Rm|\leq 4 c_a^{-2}$ in the ball $B(x_0, \frac{1}{2}c_a r)$ for every $0 \leq k \leq 5$ whenever
        $\omega_{m}^{-1}r^{-m}|B(x_0,r)| \geq 1-\delta_0$.
  \item Density estimate: $\displaystyle r^{2p_0-m} \int_{B(x_0, r)} \mathbf{vr}^{(r)}(y)^{-2p_0} dy \leq 2\mathbf{E}$.
  \item Connectivity estimate: $B(x_0,r) \cap \mathcal{F}_{\frac{1}{50}c_b r}^{(r)}(M)$ is $\frac{1}{2}\epsilon_b r$-regular-connected on the scale $r$. In other words, every two points 
     $x,y \in B(x_0,r) \cap \mathcal{F}_{\frac{1}{50}c_b r}^{(r)}(M)$ can be connected by a curve $\gamma \subset \mathcal{F}_{\frac{1}{2}\epsilon_b r}^{(r)}$ satisfying $|\gamma|<2d(x,y)$.  
\end{enumerate}
Then we define canonical radius of $x_0$ to be the supreme of all the $r_0$ with the properties mentioned above.
We denote the canonical radius by $\mathbf{cr}(x_0)$.
For subset $\Omega \subset M$, we define  the canonical radius of $\Omega$ as the infimum of all $\mathbf{cr}(x)$ where $x \in \Omega$.
We denote this canonical radius by $\mathbf{cr}(\Omega)$.
\label{dfn:SC02_1}
\end{definition}

\begin{remark}
The constants $c_a, c_b, \epsilon_b, \boldsymbol{E}$ in Definition~\ref{dfn:SC02_1} are all uniform constants depending on $n$ and $\kappa$. 
If $M \in \widetilde{\mathscr{KS}}(n,\kappa)$ and $x_0 \in M$, then all the estimates in Definition~\ref{dfn:SC02_1} are satisfied with better constants.
In particular, we have $\mathbf{cr}(x_0)=\infty$. 
\label{rmk:MD24_1}
\end{remark}

\begin{remark}
In Definition~\ref{dfn:SC02_1}, $\mathbf{vr}^{(r)}$ can be replaced by other regularity scales(c.f. the comments before Proposition 5.11 of Chen-Wang~\cite{CW6}), including harmonic radius, reduced volume radius, curvature radius, etc.
The number $p_0$ is  a number very close to $2$.  We set $p_0$ to be $2-\frac{1}{1000n}$ as in Chen-Wang~\cite{CW6}. 
\label{rmk:MD13_2}  
\end{remark}

\begin{definition}
   We call the space-time canonical radius ($\mathbf{scr}$) of point $(x_0, t_0)$ is greater than $r$ whenever $\mathbf{cr}(M, g(t_0)) \geq r$ and the following two-sided pseudo-locality is satisfied:
   
   If $\omega_{m}^{-1}r^{-m}|B(x_0,r)| \geq 1-\delta_0$, then 
   \begin{align}
     |Rm|(x,t) \leq 4c_a^{-2}r^{-2},    \label{eqn:MB04_0}
   \end{align}
   for every $(x,t) \in B(x_0, \frac{1}{2}c_a r) \times [t_0-\frac{1}{4}c_a^2 r^2, t_0+\frac{1}{4}c_a^2 r^2]$. 
   
   We say $\mathbf{scr}(M, t_0) \geq r$ if $\mathbf{scr}(x, t_0) \geq r$ for every point $x \in M$. 
\label{dfn:MA13_1}
\end{definition}

\begin{remark}
 The concept $\mathbf{scr}$ is not used in Chen-Wang~\cite{CW6}, where we have $\mathbf{pcr}$ instead. 
 Both of them are auxiliary concepts and can be dropped at the end(c.f. Remark~\ref{rmk:MD28_3}).
 In Chen-Wang~\cite{CW6}, we study flow compactness and line bundle compactness, $\mathbf{pcr}$ is needed for
 the Chen-Lu inequality to show the preserving of regularity along time direction without $|R| \to 0$.
 The lower bound of $\mathbf{pcr}$ implies a lower bound of $\mathbf{scr}$.   In other words, lower bound of $\mathbf{pcr}$ implies
 two-sided pseudo-locality theorem.  This is not an obvious result and is proved in Theorem 1.4 of  Chen-Wang~\cite{CW6}. 
 
 %
\label{rmk:MD28_1}  
\end{remark}

\begin{definition}
We say a Riemannian manifold $(M^m, g)$ is $\kappa$-noncollapsed on scale $r_0$ if for every $x \in M$ and $r \in (0,r_0)$, we have $r^{-m}|B(x,r)|>\kappa$. 
We say a Ricci flow $\mathcal{M}=\{(M^m, g(t)), t \in I\}$ is $\kappa$-noncollapsed on scale $r_0$ if each time slice $(M^m, g(t))$ is a $\kappa$-noncollapsed Riemannian manifold on scale $r_0$.
\label{dfn:MD16_1}
\end{definition}

\begin{remark}
Note that our $\kappa$-noncollapsed condition is different from the one of Perelman(c.f.~\cite{Pe1}).  In this paper, we only study Ricci flows which is $\kappa$-noncollapsed on scale $1$. 
\label{rmk:MB24_1}
\end{remark}

\section{A principle of metric distortion estimate}
This is nothing but a generalization of Section 5.3 of~\cite{CW6}, with the Bergman function there replaced by a general function $u$ whose gradient and time derivatives are bounded. 
The basic idea is to setup a uniform estimate such that the level sets of $u$ can be compared with geodesic balls of fixed sizes(c.f. Figure~\ref{fig:levelsets}).  In other words, for each fixed $a$, we have positive $r,\rho$ such that
\begin{align*}
  B_{g(t)}(x, r) \subset    \Omega_t(x, a)=\{y | u(y, t) \geq a\} \subset B_{g(t)}(x, \rho).
\end{align*}
The key is the $r,\rho$ in the above inequality does not depend on $x,t$. 
Since the time derivative of $u$ is uniformly bounded,  we can transform the comparison of geodesic balls at different time slices to the comparison of level sets of $u$
at different time slices. However, the time derivative of $u$ is uniformly bounded. Therefore, the level sets at different time slices can be compared. 
Using this principle, we have the following metric distortion estimate.

\begin{figure}
 \begin{center}
 \psfrag{A}[c][c]{$\color{blue} B(x,r)$}
 \psfrag{B}[c][c]{$\color{red} \Omega(x, a)$}
 \psfrag{C}[c][c]{$B(x, \rho)$}
 \psfrag{D}[c][c]{$x$}
 \includegraphics[width=0.5 \columnwidth]{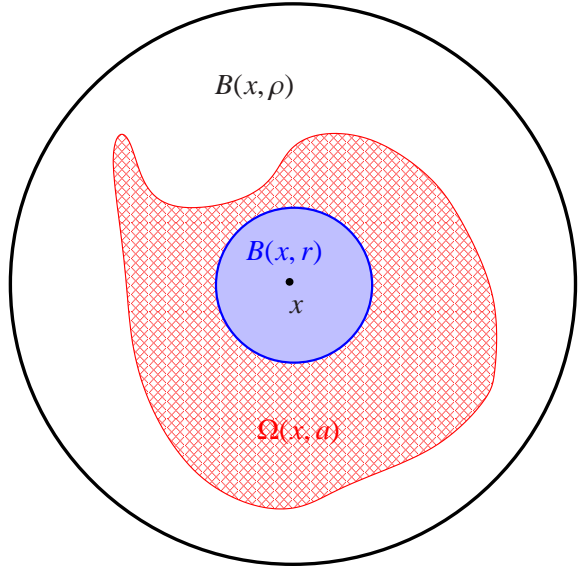}
 \caption{Compare level sets of a Lipshitz function with geodesic balls}
 \label{fig:levelsets}
 \end{center}
 \end{figure}

\begin{theorem}[\textbf{Rough metric distortion estimate}]
Let $\mathcal{M}=\{(M, g(t)), -1 \leq t \leq 1\}$ be a smooth space-time, i.e., smooth family of metrics $g(t)$ on a complete manifold $M^m$. 
Suppose $u$ is a positive function defined on $\mathcal{M}$ such that 
  \begin{align}
  \begin{cases}
     & |\dot{u}| + |\nabla u|<A, \\
     & \frac{1}{A} \leq \left. \int_M u d\mu \right|_{t} \leq A, \quad \forall \; t \in [-1, 1], \\
     & \frac{1}{A} \leq \frac{|B_{g(t)}(x, r)|}{r^m},  \quad \forall \; x \in M, \; t \in [-1, 1], r \in (0, 1).   
  \end{cases}    
  \label{eqn:MA16_4}
  \end{align}
  If $u(x_0, 0)>\frac{1}{A}$, then we have a ball containment
  \begin{align}
     B_{g(0)} \left(x_0, 0.2A^{-2} \right) \subset B_{g(t)} \left( x_0, (10A)^{2m+1} \right), \quad \forall \; t \in [-0.2A^{-2}, 0.2A^{-2}]. 
  \label{eqn:MA16_5}   
  \end{align}  
\label{thm:MA16_1} 
\end{theorem}

\begin{proof}
  Because of the gradient estimate, we see that $u(\cdot, 0)>0.5 A^{-1}$ in the ball $B_{g(0)}(x_0, 0.5 A^{-2})$. 
  Applying the time derivative estimate, we then obtain $u>0.2 A^{-1}$ for all $x \in B_{g(0)}(x_0, 0.2 A^{-2})$ and $t \in [-0.2A^{-2}, 0.2A^{-2}]$. 
  Take a point $y \in \partial B_{g(0)}(x_0, 0.2A^{-2})$. 
  Let $t>0$ be a time such that $y$ can be connected to $x_0$ by a curve $\gamma$ where $u>0.2A^{-1}$. We claim that $x_0$ and $y$ cannot be too far away.
  Actually, $\cup_{z \in \gamma} B(z, 0.1A^{-2})$ is a covering of $\gamma$, which is compact.  
  We can take a finite cover  $\displaystyle \cup_{i=1}^N B(z_i, 0.1A^{-2})$ such that $B(z_i, 0.01A^{-2})$ are disjoint, with $z_1=x_0$ and $z_N=y$. 
  Since $z_i \in \gamma$, we have $u(z_i, t)>0.2A^{-1}$, which in turn implies that $u(\cdot, t)>0.1 A^{-1}$ in the ball $B(z_i, 0.01A^{-2})$.  Then we have
  \begin{align*}
     A \geq \int_M u d\mu \geq  \int_{\cup_{i=1}^N B(z_i, 0.01A^{-2})} u d\mu  \geq 0.1 A^{-1} \sum_{i=1}^{N} |B(z_i, 0.01A^{-2})| \geq 0.1 A^{-1} \cdot  A^{-1} \cdot \left(0.01 A^{-2} \right)^m N,
  \end{align*}
  where we used (\ref{eqn:MA16_4}) in the last step.
  It follows that $\displaystyle N \leq 10^{2m+1} A^{2m+3}$.  Now using the fact that $\cup_{i=1}^N B(z_i, 0.1A^{-2})$ is a covering of $\gamma$, we obtain
  \begin{align*}
    d(x_0, y) \leq N \cdot 0.1A^{-2} \leq 10^{2m} A^{2m+1}<(10A)^{2m+1}. 
  \end{align*}
  Let $\gamma$ be the shortest geodesic connecting $x_0$ and $y$ at time $t=0$. Then for each $t \in [-0.2A^{-2}, 0.2A^{-2}]$, we have $u>0.2A^{-1}$. 
  By the previous argument, we know $d_{g(t)}(x_0, y) \leq 10^{2m}A^{2m+1}$, which implies  (\ref{eqn:MA16_5}).  
\end{proof}

\begin{remark}
Note that Theorem 3.1 plays an important role in the paper of Chen-Wang~\cite{CW6}, where $u$ is the peak section function or the exponential of the Bergman function.  
More details can be found in section 5.3(Lemma 5.20 and Lemma 5.21 in particular) of Chen-Wang~\cite{CW6}. 
On the tangent cone(local structure) level, the expression of peak section function is $e^{-|z|^2}$, which is almost the same as the heat kernel function $e^{-\frac{|z|^2}{|t|}}$,  where $|z|$ is the distance to the vertex of the cone.
In retrospect, it is not surprising that the peak section function in Chen-Wang~\cite{CW6} can be replaced by heat kernel function to extend our distance distorsion estimate to the Riemannian setting,
as done in  the later work of Bamler-Zhang~\cite{BZ}.   Such distance distorsion estimate is the basis of Bamler-Zhang~\cite{BZ} and is essentially used in Bamler~\cite{Bamler}.
\label{rmk:MD09_1}
\end{remark}

\section{Structure of Ricci flows with space-time canonical radius($\mathbf{scr}$) bounded below}
\label{sec:scrbd}
The argument in this section is purely Riemannian. 
In this section, we focus on the study of Ricci flows
\begin{align}
     \frac{\partial}{\partial t} g=-2Ric 
\label{eqn:GC21_7}     
\end{align}
on closed manifolds
$\mathcal{M}=\{(M^m, g(t)), -T \leq t \leq T\}$ satisfying
\begin{align}
\begin{cases}
  &\mathbf{scr}(M, g(t)) \geq 1, \quad \forall t \in [-T+1, T-1];\\
  &|R|(x, t) + \frac{2}{T} \leq 1, \quad \forall (x, t) \in M \times [-T+1, T-1]. 
\end{cases}  
\label{eqn:MA16_2}
\end{align}
The second inequality in the above system implies that $T \geq 2$ and $|R| \leq 1$ on the whole space-time $M \times [-T, T]$. 
Under the condition (\ref{eqn:MA16_2}),  we denote $\mathbf{vr}^{(1)}$(c.f. Definition~\ref{dfn:SC24_1}) by $\mathbf{vr}$ for simplicity of notations. 
Furthermore, for each $0<r<1$, we define
\begin{align}
\mathcal{F}_{r}(M, t) \triangleq \{x \in M|  \mathbf{vr}(x,t) \geq r\},  \qquad
\mathcal{D}_{r}(M, t) \triangleq \{x \in M|  \mathbf{vr}(x,t) < r\}. 
\label{eqn:SL27_1}
\end{align}
\textbf{Note that (\ref{eqn:GC21_7}), (\ref{eqn:MA16_2}) and (\ref{eqn:SL27_1}) are the common conditions or conventions for all the discussion in Section~\ref{sec:scrbd}.}

Note that the estimate (\ref{eqn:MA16_2}) already implies a rough weak-compactness property. 

\begin{proposition}[\textbf{Rough weak compactness of time slices}]
 Suppose $\mathcal{M}_i=\{(M_i^m, g_i(t)), -T_i \leq t \leq T_i\}$  is a sequence of Ricci flows satisfying (\ref{eqn:GC21_7}) and (\ref{eqn:MA16_2}),
 $x_i \in M_i$.
 Then by taking subsequence if necessary, we have
 \begin{align}
   (M_i,x_i, g_i(0)) \longright{\hat{C}^{\infty}}  (\bar{M},\bar{x}, \bar{g}),      \label{eqn:MD16_4}
 \end{align}
 where $\bar{M}$ is a length space with regular-singular decomposition $\bar{M}=\mathcal{R} \cup \mathcal{S}$. 
 The regular part $\mathcal{R}$ is an open, smooth Riemannian manifold.
           Furthermore, for every two points $x,y \in \mathcal{R}$, there exists a curve $\gamma$ connecting $x,y$ satisfying
           \begin{align}
           \gamma \subset \mathcal{R}, \quad  |\gamma| \leq 3d(x,y).       \label{eqn:HE11_1}
           \end{align}
 The singular part $\mathcal{S}$ satisfies the Minkowski dimension estimate
 \begin{align}
   \dim_{\mathcal{M}} \mathcal{S} \leq m-2p_0,
 \label{eqn:SB13_4}
 \end{align}         
where $p_0=2-\frac{1}{500m}$. 
\label{prn:MD16_1}
\end{proposition}

\begin{proof}
 This follows exactly from Theorem 3.18 of Chen-Wang~\cite{CW6}. Note that we have backward pseudo-locality since $\mathbf{scr}_{g_i}(M_i, 0) \geq 1$.  
 Therefore, the limit regular part $\mathcal{R}$ is a smooth manifold. 
\end{proof}

In Proposition~\ref{prn:MD16_1}, the limit space $\bar{M}$ in general does not belong to the model space $\widetilde{\mathscr{KS}}(n,\kappa)$. 
However, it does satisfy the first two defining properties of $\widetilde{\mathscr{KS}}(n,\kappa)$ in Definition~\ref{dfn:GC21_1}, except the K\"ahler condition. 
In order to show $\bar{M} \in \widetilde{\mathscr{KS}}(n,\kappa)$, one need further input to guarantee the property 3 to property 6  are all true. 
Actually, all the extra conditions needed are the following equations
\begin{align}
  \lim_{i \to \infty} \left( \frac{1}{T_i} + \frac{1}{\Vol(M_i, g_i(0))} + \sup_{M_i \times [-T_i, T_i]} |R|  \right)=0.  \label{eqn:SL06_1}
\end{align}
In the remainder part of this section, we shall show that $\bar{M}$ satisfies property 1 to property 6 of Definition~\ref{dfn:GC21_1}, except the K\"ahler condition.

\begin{lemma}[\textbf{Time derivative estimate}]
  Let $u$ be a heat solution $\square u=0$ on $M \times [s_0,T-1]$, where $s_0 \in [-T+1, T-1]$.
  Then we have
  \begin{align}
  \begin{cases}
    & |\nabla u|^2 \leq \max_{t=s_0} |\nabla u|^2, \\ 
   &|\dot{u}|=|\Delta u| \leq \max_{t=s_0} |\Delta u| + \max_{t=s_0} |\nabla u|^2 + \frac{R+1}{4}, 
  \end{cases} 
  \label{eqn:MA16_3}
  \end{align}
  for each $t \in [s_0, T-1]$. 
  \label{lma:MA13_1}
\end{lemma}

\begin{proof}
Without loss of generality, we assume $s_0=0$. 
Direct calculation shows that
\begin{align*}
  \square \Delta u=2R_{ij}u_{ij},  \quad \square |\nabla u|^2=-2|u_{ij}|^2,  \quad \square R=2|R_{ij}|^2. 
\end{align*}
Combining them in an proper way implies that
\begin{align*}
  \square \left(\Delta u -\frac{R}{4} -|\nabla u|^2 \right)=-2 \left|u_{ij}-\frac{R_{ij}}{2} \right|^2, \quad
  \square \left(-\Delta u -\frac{R}{4} -|\nabla u|^2 \right)=-2 \left|u_{ij}+\frac{R_{ij}}{2} \right|^2.
\end{align*}
Therefore, the functions $-R, |\nabla u|^2, \Delta u-\frac{R}{4}-|\nabla u|^2, -\Delta u -\frac{R}{4} -|\nabla u|^2$ are all sub-solutions of heat equation.
The first line of (\ref{eqn:MA16_3}) follows from the fact that $\square |\nabla u|^2 \leq 0$ directly by maximum principle. 
To prove the second line, we note that maximum principle implies
\begin{align*}
   \Delta u &\leq \max_{t=0} \left( \Delta u -\frac{R}{4} -|\nabla u|^2 \right) +\frac{R}{4} +|\nabla u|^2
    \leq \max_{t=0} \left( \Delta u -|\nabla u|^2 \right) - \min_{t=0} \frac{R}{4}  + \frac{R}{4} +|\nabla u|^2 \\
  &\leq \max_{t=0} \left( \Delta u -|\nabla u|^2 \right) + \max_{t=0} |\nabla u|^2 + \frac{R+1}{4}
    \leq \max_{t=0} \Delta u + \max_{t=0} |\nabla u|^2 +\frac{R+1}{4}.
\end{align*}
Recall that we used the condition $R \geq -1$ by (\ref{eqn:MA16_2}). 
Similarly, we have
\begin{align*}
  -\Delta u &\leq \max_{t=0} \left( -\Delta u -\frac{R}{4} -|\nabla u|^2 \right) + \frac{R}{4} +|\nabla u|^2 \leq \max_{t=0} \{-\Delta u-|\nabla u|^2\} + \max_{t=0} |\nabla u|^2+\frac{R+1}{4}\\
  &\leq \max_{t=0} \left( -\Delta u \right) + \max_{t=0} |\nabla u|^2 + \frac{R+1}{4}.
\end{align*}
Then the second line of (\ref{eqn:MA16_3}) follows from the combination of the  previous two estimates.   
\end{proof}

The following gradient estimate, due to Q.S. Zhang (c.f.~\cite{Zhq2}) and Cao-Hamilton (c.f.~\cite{CaHa}) will be repeatedly used. For the convenience of readers, we write down the precise statement.
We follow the style of Q.S. Zhang. 

\begin{lemma}[\textbf{Gradient estimate and Harnack inequality of heat solution}]
Let $u$ be a positive heat solution $\square u=0$ on $M \times [t_0,T]$, where $t_0 \in [-T+1, T]$.
If $u \leq A$ on $M \times [t_0, T]$, then we have
\begin{align}
    \frac{|\nabla u|}{u}  \leq \frac{1}{\sqrt{t-t_0}} \sqrt{\log \frac{A}{u}}.
\label{eqn:MA16_6}    
\end{align}
For each $t \in (t_0, T]$ and two points $x,y \in M$, we have
\begin{align}
 u(x,t)e^{-\left( \frac{2d}{\sqrt{t-t_0}} + \sqrt{\log \frac{A}{u(x,t)}}\right)^2} < u(y, t) < u(x,t)e^{\left( \frac{2d}{\sqrt{t-t_0}} + \sqrt{\log \frac{A}{u(x,t)}}\right)^2} 
\label{eqn:MA16_7} 
\end{align}
where $d=d_{g(t)}(x,y)$. 
\label{lma:MA16_1}
\end{lemma}

Using the fact $\mathbf{scr} \geq 1$, we can construct auxiliary heat solutions for the purpose of applying metric distorsion estimate Theorem~\ref{thm:MA16_1}. 
Similar choice of heat solution was already used in the proof of Lemma 4.2 of Chen-Wang~\cite{CW6}. 

\begin{lemma}[\textbf{Construction of heat solution}]
Suppose $\mathbf{vr}(x_0, t_0)=r$. Then there is a heat solution on $M \times [t_0-K^{-2}r^2, T]$ such that
\begin{align}
 &\int_M u(\cdot, t) d\mu<10,  \label{eqn:MA17_2} \\
 &\frac{K^mr^{-m}}{C}<u(x_0, t) < CK^mr^{-m},      \label{eqn:MA17_3}\\
 &|\Delta u| +|\nabla u|^2 < CK^{m+2} r^{-m-2}, \label{eqn:MA17_4}
\end{align}
for all $t \in [t_0-0.5K^{-2}r^2, t_0+ 0.5K^{-2}r^2]$. 
\label{lma:MA16_2}
\end{lemma}

\begin{figure}
 \begin{center}
 \psfrag{A}[c][c]{$\color{blue}t=t_0=0$}
 \psfrag{B}[c][c]{$\color{red}t=s_0=-K^{-2}r^2$}
 \psfrag{C}[c][c]{$\color{green}B_{g(s_0)}(x_0, K^{-1}r) \times [s_0, t_0]$}
 \psfrag{D}[c][c]{$t$}
 \psfrag{E}[c][c]{$M$}
 \psfrag{F}[c][c]{$\color{blue}B_{g(t_0)}(x_0, r)$}
 \psfrag{G}[c][c]{$\color{red} B_{g(s_0)}(x_0, K^{-1}r)$}
 \psfrag{H}[c][c]{$(x_0, t_0)$}
 \psfrag{I}[c][c]{$(x_0, s_0)$}
 \includegraphics[width=0.5 \columnwidth]{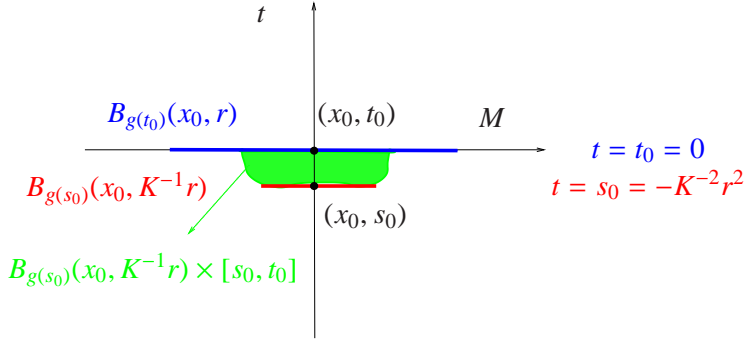}
 \caption{Different domains for heat solution construction}
 \label{fig:heatsolution}
 \end{center}
 \end{figure}

\begin{proof}
 Without loss of generality, we assume $t_0=0$. 
 
 According to (\ref{eqn:MB04_0}) in Definition~\ref{dfn:MA13_1} and our assumption (\ref{eqn:MA16_2}), 
 by adjusting $K$ if necessary, we know $|Rm| \leq 0.01 K^2 r^{-2}$ and $inj \geq 10 K^{-1}r$ inside the ball $B(x_0, K^{-1}r)$, with respect to
 the metric $g(s_0)$.  For simplicity of notation, we define $s_0=-K^{-2}r^2$. 
 Let $\eta$ be a cutoff function which equals $1$ on $(-\infty, 0.5)$ and decreases from $1$ to $0$ on $(0.5, 1)$, and vanishes on $(1, \infty)$.  
 On the time-slice $t=s_0=-K^{-2}r^2$, we set bump function $\varphi=\omega_m K^m r^{-m}\eta(\frac{Kd}{r})$ where $d$ is the distance to $x_0$ under the metric $g(s_0)$.  
 The geometry bound in the ball $B(x_0, K^{-1}r)$ then implies that $\frac{1}{2}<\int_M \varphi d\mu <2$. 
 Starting from time slice $s_0=-K^{-2}r^2$ and the bump function $\varphi$, we solve the heat equation and obtain the solution $u$.  Recall that
 \begin{align*}
     \frac{d}{dt} \int_M u d\mu=\int_M(\square u - Ru) d\mu \leq \int_M u d\mu,   \Rightarrow  \left. \int_M u d\mu \right|_{t}  \leq  \left. e^{t+K^{-2}r^2} \int_M ud\mu \right|_{s_0}< 2 e^{t+K^{-2}r^2}. 
 \end{align*}
 Plugging the fact $t \in [-0.5K^{-2}r^2, 0.5K^{-2}r^2] \subset [s_0, s_0+2 K^{-2} r^2]$ into the above inequality, we obtain (\ref{eqn:MA17_2}). 
 
 We now move on to prove (\ref{eqn:MA17_3}).   For simplicity, we first show (\ref{eqn:MA17_3}) for the time $t=t_0=0$. 
  Note that the upper bound of $u(x_0, 0)$ follows from maximum principle and the fact $\varphi \leq \omega_m K^m r^{-m}$.
 For the lower bound of $u(x_0, 0)$, we need some comparison geometry of the Ricci flow. 
 Let $l$ be the reduced distance to the space-time point $(x_0, 0)$, $\square^*=-\partial_t-\Delta+R$ be the conjugate operator of $\square$, $\tau=-t$.  Then it follows from the celebrated work of
 Perelman that
 \begin{align}
   \square^*  \{(4\pi \tau)^{-\frac{m}{2}} e^{-l} \}=(-\partial_t - \Delta + R) \{(4\pi \tau)^{-\frac{m}{2}} e^{-l} \} \leq 0    \label{eqn:MD03_1}
 \end{align}
 This inequality was written down by Perelman as inequality (7.13) and (7.15) in section 7 and Corollary 9.5 in section 9 of \cite{Pe1}. 
 Note that there is a mistake of statement in the proof of Corollary 9.5~\cite{Pe1}, which is corrected in Corollary 29.23 in Kleiner-Lott~\cite{KL}. 
 Consequently, it follows from (\ref{eqn:MD03_1}) that
 \begin{align*}
    \frac{d}{dt} \int_M u (4\pi \tau)^{-\frac{m}{2}} e^{-l} d\mu =- \int_M u \square^* \left\{(4\pi \tau)^{-\frac{m}{2}} e^{-l} \right\} d\mu  \geq 0. 
 \end{align*}
 Note that $(4\pi \tau)^{-\frac{m}{2}} e^{-l}$ converges to $\delta$-function at $(x_0, 0)$ as $\tau \to 0$. Integrating the above equation from $t=s_0=-K^{-2}r^2$ to $t=0$ implies that
 \begin{align*}
    u(x_0, 0) \geq \left. \int_{M} u \left(4\pi K^{-2}r^2 \right)^{-\frac{m}{2}} e^{-l} d\mu \right|_{t=-K^{-2}r^2}= \left. \int_{M} \varphi \left(4\pi K^{-2}r^2 \right)^{-\frac{m}{2}} e^{-l} d\mu \right|_{t=-K^{-2}r^2}.
 \end{align*}
 On the ball $B_{g(s_0)}(x_0, K^{-1}r)$, metrics $g(t)$ are all equivalent for $t \in [-K^{-2}r^2, 0]$(c.f. Figure~\ref{fig:heatsolution}).
 Then $l \leq \frac{Cd^2}{K^{-2}r^2} \leq C$ on $B_{g(s_0)}(x_0, K^{-1}r)$, where
 $d$ is the distance to $x_0$ with respect to the metric $g(s_0)$.
 One can check R.G. Ye's paper~\cite{Ye08} for more details about this estimate. 
 Consequently, we obtain
 \begin{align*}
   u(x_0, 0) \geq \int_{B(x_0, 0.5 K^{-1}r)} \varphi  \left(4\pi K^{-2}r^2 \right)^{-\frac{m}{2}} e^{-C} d\mu \geq \frac{1}{C} K^m r^{-m}, 
 \end{align*}
 which is exactly the first inequality of (\ref{eqn:MA17_3}), at time $t=t_0=0$.  It is not hard to see that the above proof of (\ref{eqn:MA17_3}) at time $t=0$ can be applied 
 for each $t \in [t_0-0.5K^{-2}r^2, t_0+0.5K^{-2}r^2]$.  Actually, the upper bound follows from maximum princple and the fact $\varphi \leq \omega_m K^m r^{-m}$ at $t=s_0=-K^{-2}r^2$.
 The lower bound follows from the upper bound of reduced distance from $(x_0, t)$ to $(x, s_0)$, for each $x \in B_{g(s_0)}(x_0, K^{-1}r)$. 
 Note that we have the reduced distance estimate since geometry are uniformly bounded on $B_{g(s_0)}(x_0, K^{-1}r) \times [t_0-K^{-2}r^2, t] \subset B_{g(s_0)}(x_0, K^{-1}r) \times [t_0-K^{-2}r^2, t_0+K^{-2}r^2]$.

 Now we prove (\ref{eqn:MA17_4}). Note that the uniform geometry bound around $(x_0, s_0)=(x_0, -K^{-2}r^{2})$ and the choice of $\varphi$ guarantees that
$|\Delta \varphi| + |\nabla \varphi|^2 < Cr^{-2}$.  Therefore, the time-derivative estimate, Lemma~\ref{lma:MA13_1} applies and we obtain
\begin{align*}
 |\nabla u|^2 < Cr^{-2}, \quad  |\Delta u| \leq Cr^{-2} + \frac{R+1}{4}< Cr^{-2}
\end{align*}
for all $t \in [s_0, T-1]$. Note that $s_0=-K^{-2} r^2$ and $[-0.5 K^{-2} r^2, 0.5 K^{-2} r^2] \subset [s_0, T-1]$. 
Therefore, (\ref{eqn:MA17_4}) is proved. 
\end{proof}

\begin{corollary}[\textbf{Volume non-inflating on large scale}]
For each $(x,t) \in M \times [-T+1,T]$ and $\rho>1$, we have 
\begin{align}
    |B(x, \rho)|< C e^{C(\rho+1)^2}   \label{eqn:MA15_0}
\end{align}
where $C=C(m, \kappa)$ is independent of $\mathbf{vr}(x,t)$. The metric $g(t)$ is the default metric in the above inequality. 
\label{cly:MA14_1}
\end{corollary}

\begin{proof}
Without loss of generality, we assume $t=0$.  Then $g(0)$ is the default metric in the following discussion. 

By density estimate, we can find a point $x_0 \in B(x, 1)$ such that $\mathbf{vr}(x_0, 0)>c$ for some small $c=c(m, \kappa)$. 
Triangle inequality implies $B(x, \rho) \subset B(x_0, \rho+1)$. Therefore, (\ref{eqn:MA15_0}) follows from the following inequality
\begin{align}
  |B(x_0, \rho)| < C e^{C\rho^2},  \quad \forall \rho>0.        \label{eqn:MD12_1}
\end{align}
We now focus on the proof of (\ref{eqn:MD12_1}). 
Since $\mathbf{vr}(x_0,0)> c$, we can construct a heat solution $u$ as in Lemma~\ref{lma:MA16_2}. 
By Lemma~\ref{lma:MA16_1},  $u$ satisfies the following inequality at time $t=0$. 
\begin{align*}
  u(y, 0)> u(x_0, 0)e^{-\left( \frac{2K\rho}{c} + \sqrt{\log \frac{CK^m}{c^mu(x_0,0)}}\right)^2}> \frac{K^m}{Cc^m}e^{-\left( \frac{2K\rho}{c} + C\right)^2}> \frac{K^m}{Cc^m} e^{-\frac{CK^2\rho^2}{c^2}},
\end{align*}
in the ball $B(x, \rho)$, where $g(0)$ is the default metric. 
Note that $C,K$ are constants depending at most on $m, \kappa$.  On the other hand, by Lemma~\ref{lma:MA16_2}, we know $\int_M ud\mu < 10$.  It follows that
\begin{align*}
 \frac{K^m}{Cc^m} e^{-\frac{CK^2\rho^2}{c^2}}|B(x_0, \rho)|<\int_{B(x_0,\rho)} ud\mu <\int_{M} ud\mu <10, \quad \Rightarrow \quad  |B(x_0, \rho)| < \frac{10C c^m}{K^m}e^{\frac{CK^2\rho^2}{c^2}}. 
\end{align*}
Then (\ref{eqn:MD12_1}) follows from the above inequality by choosing proper new $C$. 
\end{proof}

\begin{remark}
 If $\displaystyle \inf_{0<\theta<1} \mu(g(t), \theta) \geq -C$ uniformly, then a better volume ratio upper bound than (\ref{eqn:MA15_0}) can be obtained from Chen-Wang~\cite{CW5} and Q.S. Zhang~\cite{Zhq3}.
\label{rmk:MD28_2}  
\end{remark}


\begin{corollary}[\textbf{Density  estimate on large scale}]
  There is a $C=C(m,\kappa, \rho)$ such that
  \begin{align*}
     \int_{B_{g(t_0)}(x_0, \rho)} \mathbf{vr}^{-2p_0} d\mu < C.  
  \end{align*}
\label{cly:MA15_1}  
\end{corollary}

\begin{proof}
  Recall that $\mathbf{scr}(M, 0) \geq 1$ by assumption (\ref{eqn:MA16_2}).  In particular, $\mathbf{cr}(M,0) \geq 1$ by Definition~\ref{dfn:MA13_1}.  
  By the density estimate in canonical radius assumption, i.e., property 3 in Definition~\ref{dfn:SC02_1},  it is clear that $\int \mathbf{vr}^{-2p_0}$ is uniformly bounded from above for each unit ball.   Then it suffices to show that 
  $B_{g(t_0)}(x_0, \rho)$ can be covered by finite number of unit balls, and this number is uniformly bounded.  However, due to the uniform non-collapsing of each unit ball
  and the volume upper bound by Corollary~\ref{cly:MA14_1}, this follows from a standard ball-packing argument. 
\end{proof}

The following distance estimate will follow the route of section 5.3 of~\cite{CW6}. 

\begin{lemma}[\textbf{Small-scale rough distance estimate}]
If $\mathbf{vr}(x_0, t_0)=r<K^{-1}$ and $d_{g(t_0)}(x_0, y_0)<1$,  then there is a small constant $\epsilon=\epsilon(m,\kappa, r)$ and a big constant $L=L(m,\kappa, r)$ such that
\begin{align}
   d_{g(t)}(x_0, y_0)< L, \quad \forall \; t \in [t_0-\epsilon, t_0 +\epsilon].  
\label{eqn:MA13_1}   
\end{align}
\label{lma:MA13_3}
\end{lemma}
\begin{proof}
Without loss of generality, we assume $t_0=0$. 

Choose $u$ as a heat solution constructed in Lemma~\ref{lma:MA16_2}.   
By  inequality (\ref{eqn:MA16_7}) and (\ref{eqn:MA17_3}), we obtain
$u(\cdot,0)$ is uniformly bounded from below by a small constant, say $c=c(m,\kappa,r)$, 
in the ball $B_{g(0)}(x_0, 1)$.  Recall that from (\ref{eqn:MA17_4}) we have $|\dot{u}|=|\Delta u|<CK^{m+2}r^{-m-2}$ on $B_{g(0)}(x_0, 1) \times [-0.5K^{-2}r^2, 0.5 K^{-2}r^2]$. 
Therefore, for some uniformly small $\epsilon=\epsilon(m,\kappa,r)$, we have
\begin{align*}
   \int_{-\epsilon}^{\epsilon}  |\Delta u| dt <0.2 c, \quad  \Rightarrow  \quad u(y, t)>0.8 c, \quad \forall \; y \in B_{g(0)}(x_0, 1).  
\end{align*}
Then we apply the metric distortion principle, Theorem~\ref{thm:MA16_1},  for the function $u$ to obtain (\ref{eqn:MA13_1}). 
\end{proof}

\begin{lemma}[\textbf{Large-scale rough distance estimate}]
  Suppose $\mathbf{vr}(x_0,t)>r$ for all $t \in [t_0-T_0, t_0+T_0]$. 
  Then there is a uniform small constant $\epsilon=\epsilon(m,\kappa,r)$ and a big constant $L=L(m,\kappa,r)$ such that
  \begin{align}
    d_{g(t)}(x_0, y_0) <  L e^{\frac{|t-t_0|}{\epsilon}} \{d_{g(t_0)} (x_0, y_0)+1\}, \quad \forall \; t \in [t_0-T_0, t_0+T_0].
  \label{eqn:MA13_2}  
  \end{align}
  
  \label{lma:MA13_2}
\end{lemma}

Notice that the constants $L$ and $\epsilon$ are independent of $\mathbf{vr}(y_0, t)$.  However, the condition $\mathbf{vr}(x_0, t)>r$ for all $t \in [t_0-T_0, t_0+T_0]$ is quite strong. 

\begin{proof}
Recall from Definition~\ref{dfn:MA13_1} that $c_a$ is the constant appears in (\ref{eqn:MB04_0}), $\epsilon_b$ is the constant appears in the connectivity estimate.
In Lemma~\ref{lma:MA13_3}, we can shrink $\epsilon$ if necessary such that $\epsilon \leq  \frac{1}{100} \min \left\{\epsilon_b^2, c_a^2 \right\}$.
 
Without loss of generality, we assume $t>t_0$.  By adjusting $t$ a little bit if necessary,  we may assume $\frac{t-t_0}{\epsilon}$ is a positive integer and we denote it by $N$.
  For simplicity of notation, denote $\epsilon$ by $\theta_0$.  Then $t-t_0=N\theta_0$. 
  We shall deduce the estimate (\ref{eqn:MA13_2}) by setting up a ``difference equation" with each step length $\theta_0$.

  At time $k \theta_0$, we can find a point $z \in B(y_0, Kr) \cap \mathcal{F}_r$, with respect to the metric $g(k\theta_0)$.  The existence of $z$ follows from the density estimate in the canonical radius assumption. 
  Note that $Kr<1$.  So we apply Lemma~\ref{lma:MA13_3} to $z$ and $y_0$ to obtain that
  \begin{align}
    d_{g([k+1]\theta_0)}(z,y_0)<C(r).     \label{eqn:MD16_1}
  \end{align}
  Recall that both $x_0$ and $z$ locates in $\mathcal{F}_r$ at time $k\theta_0$, by connectivity estimate, one can construct a curve $\gamma$ connecting $z$ and $x_0$ at this time such that $|\gamma|<4d(z, x_0)$
  and $\gamma \subset \mathcal{F}_{\epsilon_b r}$.  By assumption inequality (\ref{eqn:MA16_2}), the estimate (\ref{eqn:MB04_0}) in Definition~\ref{dfn:MA13_1} implies that  $\gamma \subset \mathcal{F}_{\epsilon_b^2 r}$ for all time $t \in [k\theta_0, (k+1)\theta_0]$.  Since $|Ric|$ is uniformly
  bounded inside $\mathcal{F}_{\epsilon_b^2 r}$, we can easily check that 
  \begin{align}
    |\gamma|_{g([k+1]\theta_0)}< 10|\gamma|_{g(k \theta_0)}<40 d_{g(k\theta_0)}(z, x_0).   \label{eqn:MD16_2}
  \end{align}
  In light of (\ref{eqn:MD16_1}) and (\ref{eqn:MD16_2}), the triangle inequality implies that
  \begin{align*}
    d_{g([k+1]\theta_0)}(x_0, y_0) &\leq d_{g([k+1]\theta_0)}(z, y_0) + d_{g([k+1]\theta_0)}(z, x_0)
       \leq d_{g([k+1]\theta_0)}(z, y_0) + |\gamma|_{g([k+1]\theta_0)}\\
       &< C + 40 d_{g(k\theta_0)}(z, x_0)
        \leq 40 d_{g(k\theta_0)}(y_0, x_0) + 40+C.
  \end{align*}
  Let $C'=\frac{40+C}{39}$, we obtain the difference equation
  \begin{align}
        d_{g([k+1]\theta_0)}(x_0, y_0) + C' \leq 40 \left\{ d_{g(k\theta_0)}(x_0, y_0) + C' \right\}. 
  \label{eqn:MA19_1}      
  \end{align}
  Then it follows directly that
  \begin{align*}
   d_{g(t)}(x_0, y_0)+C'=d_{g(N\theta_0)}(x_0, y_0)+C' \leq 40^{N}  \left\{ d_{g(0)}(x_0, y_0) +C' \right\}. 
  \end{align*}
  Recalling that $C'=\frac{40+C}{39}<C$ and $t-t_0=N\theta_0=N\epsilon$,  the above inequality implies that
  \begin{align}
    d_{g(t)}(x_0, y_0) \leq 40^{N} \left\{ d_{g(0)}(x_0, y_0) +C \right\} \leq C 40^{N} d_{g(0)}(x_0,y_0)=C 40^{\frac{t-t_0}{\epsilon}} d_{g(0)}(x_0,y_0). 
  \label{eqn:MA19_2}  
  \end{align}
  Let $\epsilon_d=\frac{\epsilon}{\log 40}$ and denote the last constant $C$ by $L$.
  Then (\ref{eqn:MA13_2}) follows from the above inequality for the case when $t>t_0$.
  
  If $t<t_0$, we shall move back-ward. Instead of using forward pseudo-locality, we now use the backward pseudo-locality assumption(c.f. Definition~\ref{dfn:MA13_1} and inequalities (\ref{eqn:MA16_2})). 
  Now let $N=\frac{t-t_0}{\theta_0}$. The same argument before leads to a difference equation similar to (\ref{eqn:MA19_1}):
   \begin{align*}
        d_{g(-[k+1]\theta_0)}(x_0, y_0) + C' \leq 40 \left\{ d_{g(-k\theta_0)}(x_0, y_0) + C' \right\}. 
  \end{align*}
  Solving the above difference equation, we obtain the result analogue to (\ref{eqn:MA19_2}):
  \begin{align*}
    d_{g(t)}(x_0, y_0) \leq L e^{\frac{t_0-t}{\epsilon}} d_{g(0)}(x_0, y_0), 
  \end{align*}
  which is same as (\ref{eqn:MA13_2}) since $t<t_0$. 
  
\end{proof}

After we obtain the rough estimate, we shall improve them by using the condition that scalar curvature is very small. 
This is not surprising since if this quantity is precisely zero, then the flow is static, i.e., $Ric=0$ and we have everything same as the model space.
When scalar curvature is small, there should exist an ``almost" version.  Similar ideas can also be found in Chen-Wang~\cite{CW3}, \cite{CW6}, and Tian-Wang~\cite{TW}. 
For simplicity of notations, we define 
\begin{align}
 S \triangleq \sup_{M \times [-T, T]} |R|.   \quad  \label{eqn:MB03_1}
\end{align}

\begin{lemma}[\textbf{Curvature estimates}]

There exist a big constant $C=C(m,\kappa)$ and a small constant $c=c(m,\kappa)$ with the following properties. 

Suppose $\mathbf{vr}(x_0, t_0)=r$.
For every $(x,t) \in B_{g(0)}(x_0, cr) \times [-c^2 r^2, c^2 r^2]$, we have
\begin{align}
  & |Rm|(x,t) < Cr^{-2},  \label{eqn:MA18_2}\\
  & |Ric|(x,t) < Cr^{-1} \sqrt{S}, \label{eqn:MA18_3} 
\end{align}
\label{lma:MA18_1}
\end{lemma}

\begin{proof}
 The Riemannian curvature estimate (\ref{eqn:MA18_2}) follows directly from the canonical assumption and the fact $\mathbf{vr}(x_0,t_0)=r$. 
 Hence, we only need to show (\ref{eqn:MA18_3}). Without loss of generality, we choose $c=\frac{c_a}{1000m}$, the constant in Definition~\ref{dfn:SC02_1},  and rescale the flow such that $c r=1$. 
 By scaling properties, it suffices to show that
 \begin{align}
    |Ric|(x,t) \leq  C \sqrt{S}, \quad \forall \; (x,t) \in B_{g(0)}(x_0, 1) \times [-1, 1]. 
 \label{eqn:MB04_1}   
 \end{align}
 The proof (\ref{eqn:MB04_1}) is almost the same as that of Theorem 3.2 of~\cite{Wa2}.   We repeat the basic steps of the argument here for the convenience of the readers.
 In fact,  according to our choice of $c$, it is clear that 
 \begin{align}
 |Rm|(x,t) \leq \frac{1}{m^2}, \quad inj(x,t) \geq 100, \quad \forall x \in B_{g(0)}(x_0, 4), \quad t \in [-16, 16].      \label{eqn:MB04_2}
 \end{align}
 Actually, from the evolution equation of scalar curvature and Moser iteration(c.f. Theorem 3.2 of~\cite{Wa2}), we have
\begin{align}
  \sup_{B_{g(0)}(x_0, 1) \times [-1, 1]} |Ric| \leq C(m) \left \{\int_{-2}^{2} \int_{B_{g(0)}(x_0, 2)} |Ric|^2 d\mu dt \right\}^{\frac{1}{2}}.
  \label{eqn:ricbdl2}
\end{align}
Then we can follow the calculation from inequality (37) to inequality (39) of~\cite{Wa2} to obtain 
\begin{align*}
   \int_{t_1}^{t_2} \int_{B_{g(0)}(x_0, 2)} 2|Ric|^2 d\mu dt  \leq C\left\{ \left.  \int_W |R| d\mu \right|_{t_2} + \left.  \int_W |R| d\mu \right|_{t_1} + \int_{t_1}^{t_2} \int_W |R| d\mu dt \right\},
\end{align*}
where $W=B_{g(0)}(x_0, 3)$. 
Recall that $[-2, 2] \subset [t_1, t_2] \subset [-3, 3]$.
Plugging (\ref{eqn:MB03_1}) into the above inequality, we have
\begin{align*}
   \int_{-2}^{2} \int_{B_{g(0)}(x_0, 2)} 2|Ric|^2 d\mu dt \leq \int_{t_1}^{t_2} \int_{B_{g(0)}(x_0, 2)} 2|Ric|^2 d\mu dt \leq CS. 
\end{align*}
Then (\ref{eqn:MB04_1}) follows from  (\ref{eqn:ricbdl2}) and the above inequality. 
\end{proof}

\begin{remark}
The Ricci curvature estimate (\ref{eqn:MA18_3}) is the so called estimate of the type $|Ric| \leq \sqrt{|Rm||R|}$ with $|Rm|$ being regarded as $\frac{1}{r}$ for some regularity scale $r$. 
 Such notation was already used in Chen-Wang~\cite{CW6}. For example, the comments in section 5.1 of Chen-Wang~\cite{CW6}, after the Minkowski-dimension estimate of the singular set.
 This estimate originates from Wang~\cite{Wa2} and was further developed and applied in Tian-Wang~\cite{TW} and Chen-Wang~\cite{CW6}. 
 \label{rmk:MD13_1}
\end{remark}

\begin{figure}
 \begin{center}
 \psfrag{A}[c][c]{$t$}
 \psfrag{B}[c][c]{$t_2$}
 \psfrag{C}[c][c]{$t_1$}
 \psfrag{D}[c][c]{$x_0$} 
 \psfrag{E}[c][c]{Brown: $B_{g(t_2)}(x_0, \frac{r}{4K})$}
 \psfrag{F}[c][c]{Dotted blue: $B_{g(t_1)}(x_0, \frac{r}{2K})$}
 \psfrag{G}[c][c]{Black: $B_{g(t_2)}(x_0, \frac{r}{K})$}
 \psfrag{H}[c][c]{\color{green}Uniformly regular}
 \includegraphics[width=0.5 \columnwidth]{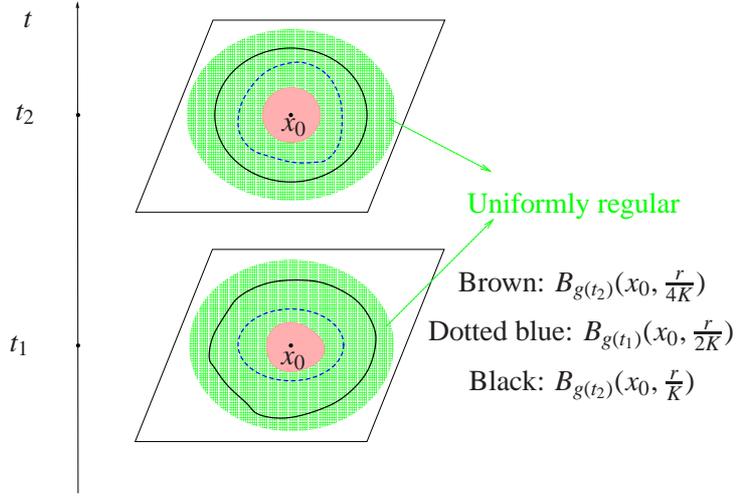}
 \caption{Precise estimate of balls on uniformly regular domains}
 \label{fig:regularcenter}
 \end{center}
 \end{figure}

\begin{lemma}[\textbf{Long-time precise distance estimate---weak version}]
There is a big constant $C=C(m,\kappa)$ with the following properties.

Suppose $\mathbf{vr}(x_0, t) \geq r$ for each $t \in [-T+1, T-1]$, then we have 
\begin{align}
   B_{g(t_2)} \left(x_0, \frac{r}{4K} \right) \subset B_{g(t_1)} \left(x_0, \frac{r}{2K} \right) \subset B_{g(t_2)}\left(x_0, \frac{r}{K} \right), \quad \forall \; t_1, t_2 \in [-T+1, T-1], 
\label{eqn:MB04_7}   
\end{align}
whenever 
\begin{align}
 S<\frac{r^2}{CT^2}.   \label{eqn:MB04_8}
\end{align}

\label{lma:MA18_2}
\end{lemma}

\begin{proof}
 Without loss of generality, we can assume $t_1<t_2$.  We shall show the second part of (\ref{eqn:MB04_7}). The proof of the first part of
 (\ref{eqn:MB04_7}) is identical and we leave it as an excercise for interested readers.  To be more precise, we want to prove
 \begin{align}
   B_{g(t_1)} \left(x_0, \frac{r}{2K} \right) \subset B_{g(t_2)}\left(x_0, \frac{r}{K} \right).  \label{eqn:MB04_4}
 \end{align}
 Suppose (\ref{eqn:MB04_4}) was wrong, then there exists $s_0 \in (t_1, t_2)$ such that
  \begin{align}
    &B_{g(t_1)} \left(x_0, \frac{r}{2K} \right) \subset B_{g(t)}\left(x_0, \frac{r}{K} \right), \quad \forall \; t \in (t_1, s_0), \label{eqn:MB04_5} \\
    & \partial B_{g(t_1)} \left(x_0, \frac{r}{2K} \right) \cap \partial B_{g(s_0)} \left(x_0, \frac{r}{K} \right) \neq \emptyset.   \label{eqn:MB04_6}
 \end{align}
 At time $s_0$, we can find a point $y_0 \in \partial B_{g(t_1)}(x_0, \frac{r}{2K}) \cap \partial B_{g(s_0)}(x_0, \frac{r}{K})$. 
 Let $\gamma$ be a shortest geodesic connecting $x_0$ and $y_0$, under the metric $g(t_1)$.  Recall that $\mathbf{vr}(x_0, t) \geq r$
 for each $t \in [-T+1, T_1]$.  Then the combination of (\ref{eqn:MB04_5}) and the Ricci curvature estimate (\ref{eqn:MA18_3}) implies that
 \begin{align*}
   |Ric|(x,t) < Cr^{-1} \sqrt{S}, \quad \forall \; x \in \gamma, \; t \in (t_1, s_0). 
 \end{align*}
 Since metrics evolove by $-2Ric$, we can estimate the length of $\gamma$ under the metric $g(s_0)$ by
 \begin{align*}
   |\gamma|_{g(s_0)} \leq e^{Cr^{-1} \sqrt{S} (s_0-t_1)} |\gamma|_{g(t_1)}= e^{Cr^{-1} \sqrt{S} (s_0-t_1)} \cdot \frac{r}{2K}. 
 \end{align*}
 On the other hand, under the metric $g(s_0)$, $\gamma$ is a curve connecting $x_0$ to a point $y_0 \in \partial B_{g(s_0)}(x_0, \frac{r}{K})$, by (\ref{eqn:MB04_6}).  
 This means that
 \begin{align*}
 |\gamma|_{g(s_0)} \geq \frac{r}{K}. 
 \end{align*}
 Combining the previous two inequalities, we have
 \begin{align*}
   Cr^{-1} \sqrt{S} (s_0-t_1) \geq \log 2,    \quad \Rightarrow \quad S \geq \frac{1}{C(s_0-t_1)^2} r^2 \geq \frac{1}{CT^2} r^2,  
 \end{align*}
 which contradicts (\ref{eqn:MB04_8}).  Therefore, (\ref{eqn:MB04_4}) holds for each $t_2 \in (t_1, T-1)$. 
\end{proof}

For static Ricci flow, the volume ratio of a fixed sized geodesic ball does not move along time.  
In the following, we shall show an ``almost" version of this property, under the condition that the scalar curvature is very small and
the center point is uniformly regular.  The key is the decomposition of space under the assumption (\ref{eqn:MA16_2}) and the delicate application of 
curvature estimates in Lemma~\ref{lma:MA18_1}.

\begin{proposition}[\textbf{Gromov-type volume estimate}]
There is an $\epsilon=\epsilon(m,\kappa,r,T_0)$ with the following properties.

Suppose $x_0 \in \mathcal{F}_{\frac{r}{K}}(M, t)$ for each $t \in [t_1,t_2] \subset [t_1, t_1+T_0] \subset [-T+1, T-1]$, $Kr<<1$.
If $x_0 \in \mathcal{F}_{Kr}(M, t_1)$, then we have
\begin{align}
    \frac{|B_{g(t_2)}(x_0, r)|_{g(t_2)} }{\left|B_{g(t_1)}\left(x_0, e^{-\frac{\delta_0}{100m}}r \right) \right|_{g(t_1)}} \geq e^{-\frac{\delta_0}{100}}, \label{eqn:MB04_9}
\end{align}
whenever $S<\epsilon$.     Similarly, if $x_0 \in \mathcal{F}_{Kr}(M, t_2)$, then we have
\begin{align}    
    \frac{|B_{g(t_1)}(x_0, r)|_{g(t_1)} }{\left|B_{g(t_2)}\left(x_0, e^{-\frac{\delta_0}{100m}}r \right) \right|_{g(t_2)}}  \geq e^{-\frac{\delta_0}{100}}, \label{eqn:MB04_10}
\end{align}
whenever $S<\epsilon$.  
\label{prn:MA18_1}
\end{proposition}

\begin{proof}
We shall first prove (\ref{eqn:MB04_9}).  Without loss of generality, we assume $t_1=0$.

Recall $K$ is a universal large constant such that whenever $x \in \mathcal{F}_r$, we have $|Rm| \leq K^2r^{-2}$, injectivity  radius bounded from below by $K^{-1}r$ in $B(x, K^{-1}r)$ and the volume ratio of the ball
$B(x, K^{-1}r)$ is at least $\omega_m \left(1-\frac{\delta_0}{100}\right)$.  Furthermore, by adjusting $K$ if necessary, we can also make the following property holds, due to the two-sided pseudo-locality assumption(c.f. Definition~\ref{dfn:MA13_1}). 
If $x \in \mathcal{F}_{Kr}(M, t)$, then $x \in \mathcal{F}_{r}(M, s)$ for each $0 \leq s \leq r^2$.  

Fix $\xi <<\frac{r}{K}$.  We shall estimate the set whose time-line $\{(x,t)| 0=t_1 \leq t \leq t_2\}$ avoids the very singular part.
 For this purpose,  we decompose the time period into $N=\frac{t_2-t_1}{\xi^2}$ equal parts with each part has length $\xi$.  Note that by adjusting $\xi$ a little bit, this can always be done. 
 For each integer $j \in \{0, 1, 2, \cdots, N\}$, we define a set $\mathbf{B}_j \triangleq \{x \in M | \mathbf{vcr}(x, j\xi^2)<K\xi\}$ at time slice $t=j\xi^2$.
 If  $y \in \mathcal{F}_{K\xi}(j\xi^2)$,  then $y \in \mathcal{F}_{\xi}(j\xi^2+s)$ for each $0 \leq s \leq \xi^2$.
 In other words, regularity scale cannot drop too quickly. 
  Then we define a ``bad subset"
  \begin{align*}
     \mathbf{B} \triangleq \left\{ \cup_{j=0}^{N} \mathbf{B}_j  \right\} \cap B_{g(0)}(x_0, r)= \cup_{j=0}^{N }  \left\{ \mathbf{B}_j \cap B_{g(0)}(x_0, r)\right\}. 
  \end{align*}
  Since $\xi <<K^{-1}r$, according to assumption, $x_0 \notin \mathbf{B}$. 
  If $\partial \mathbf{B} \cap B_{g(0)}(x_0,r)=\emptyset$, then we shall obtain the desired result directly.  
  Therefore, we assume $\partial \mathbf{B} \cap B_{g(0)}(x_0,r) \neq \emptyset$. 
  We observe that  $\partial \mathbf{B} \cap B_{g(0)}(x_0,r)$ cannot be too far away from $x_0$ under the metric $g(t)$. 
  Actually, let $y \in \partial \mathbf{B} \cap B_{g(0)}(x_0, r)$, then we have
  \begin{align*}
     y \in \mathcal{F}_{\xi}(t), \quad \forall \; t \in [0, t_2].  
  \end{align*}
  By Lemma~\ref{lma:MA18_1}, we have Ricci curvature estimate of $(y,t)$. 
  Note that $x_0 \in \mathcal{F}_{\frac{r}{K}}(M,t)$ for each $t \in [0, t_2]$. 
  Since $y \in B_{g(0)}(x_0, r) \cap \partial \mathbf{B}$, it is clear that $d_{g(0)}(x_0, y)<r$. 
  Therefore, Lemma~\ref{lma:MA13_2} applies and we have
  \begin{align}
     d_{g(t)}(y, x_0)< Ce^{\frac{t}{\epsilon_d}} \left\{ d_{g(0)}(y, x_0) +1 \right\}< 2C e^{\frac{t_2}{\epsilon_d}} \leq 2C e^{\frac{T_0}{\epsilon_d}}
  \label{eqn:MA14_1}   
  \end{align}
 By (\ref{eqn:MA18_3}) and the fact that $S$ is very small, we have
  \begin{align*}
      \int_0^{t} |Ric|(y,s)ds < C\xi^{-1} \sqrt{S} t \leq C\xi^{-1} \sqrt{S}T_0< \log 2. 
  \end{align*}
  Applying evolution equation of the Ricci flow, the above inequality implies the metric equivalence
  \begin{align}
    0.5 g(y,0) \leq e^{-CT_0 \xi^{-1} \sqrt{S}} g(y,0)  \leq  g(y, t)  \leq  e^{CT_0 \xi^{-1} \sqrt{S}} g(y, 0) \leq 2 g(y, 0),   
  \label{eqn:MA14_2}  
  \end{align}
 for $y \in \overline{B_{g(0)}(x_0,r) \backslash \mathbf{B}}$ and $t \in [0, t_2]$. 
  
  Note that $\partial \mathbf{B} \subset \cup_{j=0}^{N} \partial \mathbf{B}_j$ is $(m-1)$-dimensional, by perturbing $\xi$ if necessary. 
  We continue to estimate the $(m-1)$-dimensional Hausdorff measure of $\partial \mathbf{B}$.  Note that 
  \begin{align*}
     |\partial \mathbf{B}|_{g(0)} \leq 2^{m-1} \sum_{j=0}^{N}  \left| \partial \mathbf{B}_j  \cap B_{g(j\xi^2)}(x_0, \rho) \right|_{g(j\xi^2)},
  \end{align*} 
  where $\rho= 2C e^{\frac{T_0}{\epsilon_d}}$, as in inequality (\ref{eqn:MA14_1}).   Applying Corollary~\ref{cly:MA15_1},  we obtain for each $j$ the estimate
  \begin{align*}
     \left| \partial \mathbf{B}_j  \cap B_{g(j\xi)}(x_0, \rho) \right|_{g(j\xi)} < C \xi^{2p_0-1}.
  \end{align*} 
  Summing all the $j$-terms together and recalling $N=\frac{t_2}{\xi^2}$, we obtain
  \begin{align}
    |\partial \mathbf{B}|_{g(0)}< Ct_2 \xi^{2p_0-3} \leq CT_0 \xi^{2p_0-3}.
  \label{eqn:MA16_1}  
  \end{align}  
  According to the conditions of this proposition, $x_0 \in \mathcal{F}_{\frac{r}{K}}(M, t)$ for all $t \in [0, T_0]$. 
  By Lemma~\ref{lma:MA18_2}, we have
  \begin{align*}
     B_{g(0)}\left(x_0, \frac{r}{2K^2} \right) \cap \mathbf{B} =\emptyset,
  \end{align*}
  whenever $S$ is very small. Note that in the ball $B_{g(0)}(x_0, K^{-2}r)$, we have curvature bound $|Rm| \leq K^4 r^{-2}$.  
  Let $\Omega_{\xi}$ be the subset of the unit sphere of $T_{x_0} M$ (with respect to the metric $g(0)$) such that the unit speed geodesics coming out of
  $\Omega_{\xi}$ hit some points in $\mathbf{B}$ before it escape $B_{g(0)}(x_0, r)$. 
  Recall that by Ricci curvature estimate (\ref{eqn:MA18_3}), we have $|Ric|<Cr^{-1} \sqrt{S}<\frac{1}{100}$ since $S$ is chosen very small. 
  By Gromov-Bishop comparison,  weighted area element is almost non-increasing along each shortest geodesic.  It follows that
  \begin{align}
     |\Omega_{\xi}|_{\mathcal{H}^{m-1}}  \leq C \left( \frac{\xi}{r}\right)^{2p_0-3}.       \label{eqn:MA15_2}
  \end{align}
  Therefore, for every direction $\theta \in S^{m-1} \backslash \Omega_{\xi}$, the geodesic (at time $t=0$) starting from $\theta$ will continue to 
  $\partial B_{g(0)}(x_0, r)$ without hitting $\mathbf{B}$.  In other words, if we denote the unit speed geodesic starting from $\theta$ by $\gamma_{\theta}$, then we have
  $\gamma_{\theta}(s) \notin \mathbf{B}$ for each $s \in [0, r]$.   By the precise metric equivalence estimate (\ref{eqn:MA14_2}), we have
  \begin{align*}
    |\gamma_{\theta}|_{g(t_2)} \leq e^{Ct_2 \xi^{-1} \sqrt{S}} |\gamma_{\theta}|_{g(0)} \leq e^{CT_0 \xi^{-1} \sqrt{S}} |\gamma_{\theta}|_{g(0)}=e^{CT_0 \xi^{-1} \sqrt{S}}L. 
  \end{align*}
  Define a set $D \triangleq \left\{x \in B_{g(0)}(x_0,r) \left| x=\gamma_{\theta}(s) \; \textrm{for some} \; \theta \in S^{m-1} \backslash \Omega_{\xi}, \; s \in \left[0,e^{-CT_0 \xi^{-1} S}r \right]  \right.\right\}$. Then we have
  \begin{align}
     D \subset B_{g(T_0)} \left( x_0,  r \right).   \label{eqn:MB04_11}
  \end{align}
  Recall that $t_1=0$ and we have $x_0 \in \mathcal{F}_{Kr}(M, 0)$.  It follows from the fact $Kr<<1$, $\mathbf{scr} \geq 1$ and the improving regularity property in Definition~\ref{dfn:SC02_1} that we have curvature bound
  \begin{align*}
      |Rm|(x,0) \leq  r^{-2}, \quad \forall \; x \in B_{g(0)}(x_0, r). 
  \end{align*}
  By (\ref{eqn:MA15_2}) and the above curvature bound at time $t=0$, we see that 
  \begin{align}
    |D|_{g(0)} \geq \left|B_{g(0)} \left(x_0, e^{-CT_0 \xi^{-1} \sqrt{S}}r \right) \right|_{g(0)}- Cr^{3-2p_0}\xi^{2p_0-3}.     \label{eqn:MB04_12}
  \end{align}
 Since the volume element evolves by $-R$,  it follows from $|R|\leq S$ that
  \begin{align}
    |D|_{g(t_2)} \geq e^{-CT_0 \xi^{-1}S} |D|_{g(0)}.    \label{eqn:MB04_13}
  \end{align}
 Combining (\ref{eqn:MB04_11}), (\ref{eqn:MB04_12}) and (\ref{eqn:MB04_13}), we obtain
  \begin{align}
    \left|  B_{g(t_2)} \left( x_0,  r \right) \right|_{g(t_2)} &\geq |D|_{g(t_2)} \geq e^{-CT_0 \xi^{-1}S} |D|_{g(0)} \notag\\
    &\geq e^{-CT_0 \xi^{-1}S} \left\{ \left|B_{g(0)} \left(x_0, e^{-CT_0 \xi^{-1} \sqrt{S}}r \right) \right|_{g(0)}- Cr^{3-2p_0}\xi^{2p_0-3} \right\}.
  \label{eqn:MB04_14}  
  \end{align}
  We can choose $\xi$ small such that $Cr^{3-m-2p_0}\xi^{2p_0-3} << \frac{\delta_0^m}{1000^m}$, then choose $S$ very small such that
  $CT_0 \xi^{-1} \sqrt{S} <<\frac{\delta_0}{100m}$.  
  Note that
  \begin{align}
    &\quad \left|B_{g(0)} \left(x_0, e^{-CT_0 \xi^{-1} \sqrt{S}}r \right) \right|_{g(0)} -Cr^{3-2p_0}\xi^{2p_0-3} \notag\\
    &=\left|B_{g(0)} \left(x_0, e^{-\frac{\delta_0}{100m}}r \right) \right|_{g(0)} 
    + \left| B_{g(0)} \left(x_0, e^{-CT_0 \xi^{-1} \sqrt{S}}r \right) \backslash B_{g(0)} \left(x_0, e^{-\frac{\delta_0}{100m}}r \right) \right|_{g(0)}-Cr^{3-2p_0}\xi^{2p_0-3} \notag\\
    &\geq \left|B_{g(0)} \left(x_0, e^{-\frac{\delta_0}{100m}}r \right) \right|_{g(0)}.    \label{eqn:MB04_15}
  \end{align}
  Let us explain why the last inequality in (\ref{eqn:MB04_15}) holds. Actually,  since $Cr^{3-m-2p_0}\xi^{2p_0-3} << \frac{\delta_0^m}{1000^m}<\frac{\delta_0}{1000m}$, 
  we know the set $B_{g(0)} \left(x_0, e^{-CT_0 \xi^{-1} \sqrt{S}}r \right)  \backslash B_{g(0)} \left(x_0, e^{-\frac{\delta_0}{100}}r \right)$ contains a geodesic
  ball of radius $\frac{\delta_0}{1000m}$, whose volume is at least $\kappa \left( \frac{\delta_0 r}{1000m}\right)^m$.
  Therefore, the sum of last two terms in the line above (\ref{eqn:MB04_15}) is positive. 
  Consequently, (\ref{eqn:MB04_9}) follows from the combination of (\ref{eqn:MB04_14}) and (\ref{eqn:MB04_15}).
  
  The proof of (\ref{eqn:MB04_10}) is almost the same as (\ref{eqn:MB04_9}). For this reason, we only sktech the basic steps.
  We define a ``bad set" $\tilde{\mathbf{B}}$ at time $t_2$ as the ``projection" of the bad sets $\tilde{\mathbf{B}}_j$'s, depending on the ``step length" $\xi$.
  Applying Lemma~\ref{lma:MA13_2} backwardly, we obtain $d_{g(t_2)}(x_0, \partial \tilde{\mathbf{B}})$ is uniformly bounded. 
  Due to the density estimate in Corollary~\ref{cly:MA15_1}, we then obtain that the ``projection" of $\partial \tilde{\mathbf{B}}$ onto $B_{g(t_2)}(x_0,1)$ is
  an $(m-1)$-dimensional set whose $(m-1)$-dimensional Hausdorff measure is bounded by $Cr^{3-2p_0}\xi^{2p_0-3}$. 
  According to Bishop-Gromov comparison, this means that ``most" shortest geodesics starting from $x_0$ at time $t_2$ will be away from the ``bad set" for each $t \in [t_1, t_2]$.
  Such geodesics can be detected at time $t_1$ with length almost not changed and consequently are contained in the almost unit geodesic ball at time $t_1$. 
  Same as the steps from (\ref{eqn:MA15_2}) to (\ref{eqn:MB04_15}), we obtain that the volume ratio of unit geodesic ball at time $t_1$ is at least almost the same as
  the volume ratio of unit geodesic ball at time $t_2$. To be more precise, we have (\ref{eqn:MB04_10}) holds. 
   
\end{proof}

\begin{remark}[\textbf{Why high codimension is important}]
From the proof of Proposition~\ref{prn:MA18_1}, in particular inequality (\ref{eqn:MA16_1}), the condition $p_0>1.5$ is essentially used($p_0>1$ is good with further efforts).
Therefore, the singular set in each conifold has high Minkowski codimension is of essential importance. 
However, in the K\"ahler case, the codimension-four is automatic.  
In the real case, it holds due to the work of Cheeger-Naber~\cite{CN2}. 
\label{rmk:GC23_2}
\end{remark}

\begin{proposition}[\textbf{Weak long-time two-sided pseudo-locality}]
There is an $\epsilon=\epsilon(m,\kappa,r,T_0)$ with the following properties.

Suppose $\mathcal{M}$ is a Ricci flow solution satisfying (\ref{eqn:GC21_7}) and  (\ref{eqn:MA16_2}). 
Suppose $0<T_0<T-1$.  Then we have
 \begin{align}
    \mathcal{F}_{Kr}(M, 0) \subset \bigcap_{-T_0 \leq t \leq T_0}  \mathcal{F}_{\frac{r}{K}}(M, t)   \label{eqn:GC21_9}
 \end{align}
 whenever $S<\epsilon$.
 \label{prn:GC21_10}
\end{proposition}

\begin{proof}
Fix $x_0 \in \mathcal{F}_{Kr}(M, 0)$.
Let $b$ be the last time such that $\displaystyle x_0 \in \bigcap_{-b \leq t \leq b} \mathcal{F}_{\frac{r}{K}}(M, t)$.   It suffices to show that $b \geq T_0$.
 Actually, if $b<T_0$, then by  inequality (\ref{eqn:MB04_9}) in Proposition~\ref{prn:MA18_1}, we have 
 \begin{align*}
   \left|B_{g(b)}\left(x_0, r \right)\right|_{g(b)} \geq e^{-\frac{\delta_0}{100}} \left| B_{g(0)} \left( x_0, e^{-\frac{\delta_0}{100}} r\right)\right|_{g(0)}
   \geq \left( 1-\frac{\delta_0}{50}\right) \omega_m r^m. 
 \end{align*}
 Recall that $\mathbf{vr}(x_0,0)>Kr$. Therefore, we have $|B_{g(0)}(x_0,\rho)|_{g(0)} \geq \left( 1-\frac{\delta_0}{100}\right) \omega_m \rho^m$ for each $\rho \in (0,r)$.
 Consequently, we have
 \begin{align*}
   \omega_m^{-1}r^{-m}\left|B_{g(b)}\left(x_0, r \right)\right|_{g(b)} \geq e^{-\frac{\delta_0}{100}}  \cdot \left( 1-\frac{\delta_0}{100}\right) \cdot e^{-\frac{m\delta_0}{100m}}
   =e^{-\frac{\delta_0}{50}} \cdot \left( 1-\frac{\delta_0}{100}\right)> 1-\frac{\delta_0}{10}. 
 \end{align*}
 Similarly, using inequality (\ref{eqn:MB04_10}) in Proposition~\ref{prn:MA18_1}, we obtain
 \begin{align*}
   \omega_m^{-1}r^{-m}\left|B_{g(-b)}\left(x_0, r \right)\right|_{g(b)} > 1-\frac{\delta_0}{10}. 
 \end{align*}
 It follows that $\mathbf{vr}(x_0, b)>r$ and $\mathbf{vr}(x_0, -b)>r$.
 By continuity of volume of geodesic balls in each Ricci flow, we know $\mathbf{vr}(x_0, t)>r$ for a short time
 period beyond $[-b, b]$.
 This contradicts our assumption that $b$ is the last time that  $\displaystyle x_0 \in \bigcap_{-b \leq t \leq b} \mathcal{F}_{\frac{r}{K}}(M, t)$. 
 Therefore, we have $b \geq T_0$, which means that (\ref{eqn:GC21_9}) holds. 
 \end{proof}

Proposition~\ref{prn:GC21_10} can be regarded as a weak two-sided long-time pseudo-locality.
Here we use ``weak" because
the long-time regularity preserving only holds when the scalar curvature tends to zero. 

\begin{proposition}[\textbf{Limit space-time with static regular part}]
 Suppose $\mathcal{M}_i$  is a sequence of Ricci flows satisfying (\ref{eqn:GC21_7}) and (\ref{eqn:MA16_2}) and
 $\displaystyle S_i=\sup_{M_i \times [-T_i, T_i]} |R| \to 0$ for some $q>0$, $x_i \in M_i$  and
 $\displaystyle \lim_{i \to \infty} \mathbf{vr}(x_i,0)>0$.
 Then by taking subsequence if necessary, we have
 \begin{align}
   (M_i,x_i, g_i(0)) \longright{\hat{C}^{\infty}}  (\bar{M},\bar{x}, \bar{g}).     \label{eqn:MB09_8}
 \end{align}
 Moreover,  we have
 \begin{align}
   (M_i,x_i, g_i(t)) \longright{\hat{C}^{\infty}}  (\bar{M},\bar{x}, \bar{g})      \label{eqn:MB09_9}
 \end{align}
 for every $t \in (-\bar{T}, \bar{T})$, where $\displaystyle \bar{T}=\lim_{i \to \infty} T_i>0$.
 In particular, the limit space does not depend on time.
\label{prn:SL04_1}
\end{proposition}

\begin{proof}
By direct application of Proposition~\ref{prn:GC21_10} and the fact that $\bar{M}$ is the closure of $\mathcal{R}(\bar{M})$, we obtain the convergence limits in (\ref{eqn:MB09_8}) and (\ref{eqn:MB09_9}) are the same.
Therefore, we only need to show $\bar{M} \in \widetilde{\mathscr{KS}}(n,\kappa)$. However, after we obtain the fact that the limit does not depend on time, the proof of $\bar{M} \in \widetilde{\mathscr{KS}}(n,\kappa)$
is exactly the same as that in section 4.3 of Chen-Wang~\cite{CW6}.  We remark that the metric structure of $\bar{M}$ is totally determined by its regular part, due to the high codimension of $\mathcal{S}$ and the rough
connectivity of $\mathcal{R}$. Moreover details can be found in Proposition 4.15 of section 4.3 of Chen-Wang~\cite{CW6}. 
\end{proof}

In Proposition~\ref{prn:SL04_1},
we show that the limit flow exists and is static in the regular part, whenever we have $S \to 0$.
It is possible that the limit points in the singular part $\mathcal{S}$ are moving as time evolves.  
In general, it is much harder to estimate the movement of points in $\mathcal{S}$. 
However, for many applications, the movement of points $\mathcal{S}$ does not matter, due to the high-codimension of $\mathcal{S}$.
The regular part is a static Ricci flow solution. This fact is enough for us to apply the monotonicity of the Ricci flow to improve the regularity of the limit space. 

\begin{corollary}[\textbf{Rough estimate of reduced distance}]
Suppose $x,y \in \mathcal{R}(\bar{M})$.  Then $(x, 0)$ and $(y, -1)$ have the reduced distance bound
\begin{align*}
   l((x,0), (y,-1))<  100 d^2(x,y).
\end{align*}
\label{cly:MD03_1}
\end{corollary}
\begin{proof}
 Note that $\mathcal{R}(\bar{M}) \times (-\infty, \infty)$ is a static Ricci flow solution.   By Proposition~\ref{prn:MD16_1}, we find a curve $\gamma \subset \mathcal{R}(\bar{M})$ and $|\gamma|<3d(x,y)$.
 We can parametrize $\gamma$ such that $\gamma(0)=x$, $\gamma(1)=y$ and $|\dot{\gamma}| \equiv 3d(x, y)$. 
 This curve can be lifted to a space-time curve $\boldsymbol{\gamma}$ connecting $(x, 0)$ and $(y, -1)$, by setting $\boldsymbol{\gamma}(\tau)=(\gamma(s), -\tau)$. 
 Then we have
 \begin{align*}
   \mathcal{L}(\boldsymbol{\gamma})=\int_{0}^{1} \sqrt{\tau} (R+|\dot{\gamma}|^2) d\tau=\int_{0}^{1} \sqrt{\tau} |\dot{\gamma}|^2 d\tau \leq \frac{2}{3} \cdot 9d^2=6d^2, 
 \end{align*}
 which implies that
 \begin{align*}
   l((x,0), (y,-1)) \leq \frac{1}{2 \cdot \sqrt{1}}  \mathcal{L}(\boldsymbol{\gamma}) \leq  3 d^2(x,y) < 100 d^2(x,y). 
 \end{align*}
 Further details of more general case can be found in Lemma 4.21 of Chen-Wang~\cite{CW6}. 
\end{proof}

\begin{lemma}[\textbf{Most shortest reduced geodesics avoid high curvature part}]
For every group of numbers $0<\xi<\eta<1<H$, there is a big constant $C=C(n,A,\eta,H)$
and a small constant $\epsilon=\epsilon(n,A,H,\eta,\xi)$ with the following properties.

Let $\Omega_{\xi}$ be the collection of points $z \in M$ such that there exists a
shortest reduced geodesic $\boldsymbol{\beta}$ connecting $(x,0)$ and $(z,-1)$ satisfying
\begin{align}
   \beta \cap \mathcal{D}_{\xi}(M,0) \neq \emptyset.
\label{eqn:SL24_4}
\end{align}
Then
\begin{align}
|B_{g(0)}(x,H) \cap \mathcal{F}_{\eta}(M,0) \cap \Omega_{\xi}| < C \xi^{2p_0-1}
\label{eqn:SL24_5}
\end{align}
whenever $S<\epsilon$.
\label{lma:SL14_1}
\end{lemma}

\begin{proof}
 Same as Lemma 4.22 of Chen-Wang~\cite{CW6}. 
\end{proof}

\begin{lemma}[\textbf{Rough weak convexity by reduced geodesics}]
Suppose $\{(M_i^m, g_i(t)), -T_i \leq t \leq T_i\}$ is a sequence of Ricci flows satisfying (\ref{eqn:SL06_1}). 
Suppose $x_i \in M_i$.
Let $(\bar{M}, \bar{x}, \bar{g})$ be the limit space of $(M_i, x_i, g_i(0))$, $\mathcal{R}$ be the regular part of $\bar{M}$ and $\bar{x} \in \mathcal{R}$.
Suppose $\bar{t}<0$ is a fixed number. Then every $(\bar{z}, \bar{t})$ can be connected to $(\bar{x}, 0)$ by a smooth reduced geodesic, whenever $\bar{z}$ is away from a closed measure-zero set.
\label{lma:SK27_4}
\end{lemma}
\begin{proof}
 Same as Lemma 4.23 of Chen-Wang~\cite{CW6}. It follows directly from the application of the estimates in Lemma~\ref{lma:SL14_1}. 
\end{proof}

\begin{proposition}[\textbf{Weak convexity by Riemannian geodesics}]
Same conditions as in Lemma~\ref{lma:SK27_4}.
Then away from a measure-zero set, every point in $\mathcal{R}$ can be connected to $\bar{x}$ with a unique smooth shortest geodesic. Consequently, $\mathcal{R}$ is weakly convex.
\label{prn:SC30_1}
\end{proposition}

\begin{proof}
  Same as Proposition 4.25 of Chen-Wang~\cite{CW6}.  The weak convexity of $\mathcal{R}$ by Riemannian geodesics originates from weak convexity of $\mathcal{R} \times (-1, 0)$ by reduced geodesic, i.e., Lemma~\ref{lma:SK27_4}. 
\end{proof}

With the ``almost scalar-flat" condition (\ref{eqn:SL06_1}), we can improve the regularity of limit space $\bar{M}$ in (\ref{eqn:MD16_4}) of Proposition~\ref{prn:MD16_1}.

\begin{proposition}[\textbf{Metric structure of a blowup limit}]
Suppose $\{(M_i^{m}, x_i,  g_i(t)), -T_i \leq t \leq T_i\}$ is a sequence of Ricci flows satisfying (\ref{eqn:MA16_2}) and (\ref{eqn:SL06_1}). 
Let $(\bar{M}, \bar{x}, \bar{g})$ be the limit space of $(M_i, x_i, g_i(0))$. Then  $\bar{M}$ satisfies all the 6 defining properties of $\widetilde{\mathscr{KS}}(n,\kappa)$ except the K\"ahler condition. 
\label{prn:MD03_1}
\end{proposition}

\begin{proof}
 Same as Theorem 4.31 of Chen-Wang~\cite{CW6}, where the full details are provided.
 Here we only sketch the key point.
 The proof consists of checking all the 6 defining properties of the model space  $\widetilde{\mathscr{KS}}(n,\kappa)$. 
 In particular, the properties 3, 4 and 5 in Definition~\ref{dfn:SC24_1} are crucial.   We remark that property 3, the weak convexity of $\mathcal{R}$, follows from Proposition~\ref{prn:SC30_1}.
 Property 4, the high codimension of $\mathcal{S}$, follows from Proposition~\ref{prn:MD16_1}.  
 Property 5, the gap between regular and singular property, follows from the coincidence of volume density and reduced volume density on infinitesimal level(c.f. Theorem 2.63 of Chen-Wang~\cite{CW6}),  and the monotonicity of reduced volume.
\end{proof}

\begin{corollary}(c.f. Proposition 4.19 of Chen-Wang~\cite{CW6})
 Let $\bar{M}=\mathcal{R} \cup \mathcal{S}$ be the limit space in (\ref{eqn:MD16_4}) of Proposition~\ref{prn:MD16_1}.   Then every tangent space of $\bar{M}$ is a metric cone and $ \dim_{\mathcal{H}} \mathcal{S} \leq m-4$.
\label{cly:MD24_1} 
\end{corollary}

 \begin{proof}
  The metric cone property follows from the monotonicity of reduced volume or Perelman's local W-functional, as done in Theorem 4.18 of Chen-Wang~\cite{CW6}. 
  By the metric cone property, the Hausdorff dimension of singularity is an integer number and satisfies (\ref{eqn:SB13_4}). Consequently, we have $\dim_{\mathcal{H}} \mathcal{S} \leq m-4$. 
 \end{proof}

\section{A priori estimate of $\mathbf{scr}$}

\begin{proposition}[\textbf{Weak continuity of canonical radius and space-time canonical radius}]
Suppose $\{(M_i^{n}, x_i,  g_i(t)), -T_i \leq t \leq T_i\}$ is a sequence of K\"ahler Ricci flows satisfying (\ref{eqn:MA16_2}) and (\ref{eqn:SL06_1}). 
Then we have
\begin{align}
 &\lim_{i \to \infty} \mathbf{cr}(M_i, 0)=\infty,  \label{eqn:MD03_2}\\
 &\lim_{i \to \infty} \mathbf{scr}(M_i, 0)=\infty.  \label{eqn:MD03_3}   
\end{align}
\label{prn:MD03_2}
\end{proposition}

\begin{proof}
  The statement and proof of (\ref{eqn:MD03_2}) is exactly the same as that of Theorem 4.39 of Chen-Wang~\cite{CW6}. 
  The key is to use the a priori estimate in the model space $\widetilde{\mathscr{KS}}(n,\kappa)$ and the Cheeger-Gromov convergence to
  improve the originally assumed estimates for each $M_i$.   
  Note that each $M_i^n$ is a K\"ahler manifold of complex dimension $n$, we know from Proposition~\ref{prn:MD03_1} that the limit space $\bar{M}$ must locate in the model space $\widetilde{\mathscr{KS}}(n,\kappa)$. 
  
   Applying the same idea, we prove (\ref{eqn:MD03_3}).  Suppose (\ref{eqn:MD03_3}) fails, then by taking subsequence if necessary, we can find $x_i \in M_i$ such that
   \begin{align*}
      \lim_{i \to \infty}  \mathbf{scr}(M_i, 0) < 0.5 D<\infty
   \end{align*}
   for some $D>1$.  
   Since each $M_i \times [-T_i, T_i]$ is a compact space-time, we can find $x_i \in M_i$ such that $\mathbf{scr}(x_i, 0) \leq \mathbf{scr}(M_i, 0)+1$.  It follows that
   \begin{align}
      \lim_{i \to \infty}  \mathbf{scr}(x_i, 0) < D <\infty.   \label{eqn:MD20_1}
   \end{align}
   Note that $\mathbf{cr}(x_i, 0) \geq \mathbf{cr}(M_i, 0)$, which is very large by (\ref{eqn:MD03_2}).   This means for some $r_i \in (0, D]$ satisfying $\omega_{m}^{-1}r_i^{-m}|B(x_i,r_i)|_{g_i(0)} \geq 1-\delta_0$, 
   we do not have (\ref{eqn:MB04_0}).  In other words, we can find $y_i \in B_{g_i(0)}(x_i, r_i)$ and $t_i \in [-\frac{1}{4}c_a^2 r_i^2, \frac{1}{4}c_a^2 r_i^2]$ such that 
   \begin{align*}
      |Rm|(y_i,t_i) > 4c_a^{-2}r_i^{-2}.
   \end{align*}
   Let $\tilde{g_i}(t)=r_i^{-2} g_i(r_i^2 t)$. We have a sequence of Ricci flows $\{(M_i^n, x_i, \tilde{g}_i(t)), -r_i^{-2}T_i \leq t \leq r_i^{-2} T_i \}$ satisfying the following properties.
   \begin{itemize}
   \item $|B_{\tilde{g}_i(0)}(x_i, 1)| \geq (1-\delta_0) \omega_m$. 
   \item  For some $y_i \in B_{\tilde{g}_i(0)}(x_i, 1)$ and $s_i \in [-\frac{1}{4}c_a^2, \frac{1}{4} c_a^2]$, we have
   \begin{align}
     |\widetilde{Rm}|(y_i, s_i) > 4c_a^{-2}.   \label{eqn:MD20_2}
   \end{align}
   \end{itemize}
   Note that $x_i$ are uniformly regular, with respect to the metric $\tilde{g}_i(0)$. 
   By Proposition~\ref{prn:MD03_1} and the convergence of K\"ahler structure on the regular part,  we have convergence
   \begin{align*}
   (M_i,x_i, \tilde{g}_i(t)) \longright{\hat{C}^{\infty}}  (\bar{M},\bar{x}, \bar{g}) 
   \end{align*}
   for some $\bar{M} \in \widetilde{\mathscr{KS}}(n,\kappa)$.   Moreover, we have $|B(\bar{x}, 1)|_{\bar{g}} \geq (1-\delta_0) \omega_{2n}$. 
   Let $y_{\infty}$ be the limit point of $y_i$, under the convergence with respect to $\tilde{g}_i(0)$, $s_{\infty}$ be the limit of $s_i$.
   We remind the readers that we may have taken subsequence again.   
   In light of the a priori estimate in $\widetilde{\mathscr{KS}}(n,\kappa)$ and Proposition~\ref{prn:SL04_1}, we obtain
   \begin{align*}
       |\widetilde{Rm}|(y_{\infty}, s_{\infty})=|\widetilde{Rm}|(y_{\infty}, 0)<c_a^{-2}. 
   \end{align*}
   By smooth convergence around $(y_{\infty}, s_{\infty})$,   the above inequality means that 
   \begin{align*}
       |\widetilde{Rm}|(y_i, s_i)<c_a^{-2}, 
   \end{align*}
   for large $i$.  This contradicts (\ref{eqn:MD20_2}).
\end{proof}

\begin{theorem}[\textbf{Local space-time structure theorem}]
There is a small constant $\epsilon=\epsilon(m, \kappa)$ with the following properties.
  
  Suppose $\mathcal{M}=\{(M^n, g(t)),  -T \leq t \leq T\}$ is an unnormalized K\"ahler Ricci flow solution on a closed K\"ahler manifold $M$.
  Suppose $\mathcal{M}$ is $\kappa$-noncollapsed on a scale $r_0$ satisfying
   \begin{align}  
     |R|(x,t) + \frac{2}{T} \leq r_0^{-2}, \quad \forall x \in M,  \quad t \in [-T, T].    \label{eqn:MB07_2a}
   \end{align}  
  Then we have $\displaystyle \mathbf{scr}(M, t)>\epsilon r_0$ for each $t \in [-T+r_0^2, T-r_0^2]$.
\label{thm:MA18_1}  
\end{theorem}

\begin{proof}
Without loss of generality, we can assume $r_0=1$.

We argue by contradiction. 
If the statement was wrong, then we can find a sequence of Ricci flows $\mathcal{M}_i$'s which satisfy (\ref{eqn:MB07_2a}) and $\mathbf{scr}(M_i, t_i)=\epsilon_i \to 0$ for some $t_i \in [-T_i+1, T_i-1]$.
Fix $L>0$. 
By rearrangement of $t_i$ if necessary, we can assume further that
\begin{align*}
   \mathbf{scr}(M_i,t) \geq 0.5 \epsilon_i, \quad \forall \; t \in [t_i-L\epsilon_i^2, t_i+L\epsilon_i^2] \subset (-T_i, T_i). 
\end{align*}
It follows from the definition of $\mathbf{scr}$ that there is a point $x_i \in M_i$ such that
\begin{align*}
 \epsilon_i \leq  \rho_i=\mathbf{scr}(x_i,t_i) \leq 2 \epsilon_i.
\end{align*}
Now we rescale $g_i$ to $\tilde{g}_i$ by setting $\tilde{g}_i(t)=4\rho_i^{-2} g_i(0.25 \rho_i^2 t + t_i)$.   For simplicity, denote the space-time canonical radius with respect to $\tilde{g}_i$ by $\widetilde{\mathbf{scr}}$.
Then we have
\begin{align*}
&\widetilde{\mathbf{scr}}(x_i, 0)=2; \\
&\widetilde{\mathbf{scr}}(x, t) \geq 1, \quad \forall \; x \in M_i, \; t \in [-4L,4L].  
\end{align*}
Then we let $L \to \infty$. By taking subsequence and reordering if necessary, we obtain a sequence of Ricci flows $\{(M_i, \tilde{g}_i(t)), - 2^{i} \leq t \leq  2^i \}$ such that
\begin{align}
&\widetilde{\mathbf{scr}}(x_i, 0)=2; \label{eqn:MB09_6}\\
&\widetilde{\mathbf{scr}}(x, t) \geq 1, \quad \forall \; x \in M_i, \; t \in [-2^{i}, 2^{i}].   \label{eqn:MB09_7}
\end{align}
Then (\ref{eqn:MA16_2}) holds for $T_i=2^{i}$.  Furthermore, (\ref{eqn:SL06_1}) is satisfied by $\{(M_i, \tilde{g}_i(t)), -2^i \leq t \leq 2^i\}$.  Therefore, we can apply Proposition~\ref{prn:MD03_2} to obtain that
\begin{align*}
   \lim_{i \to \infty} \widetilde{\mathbf{scr}}(x_i, 0)=\infty,
\end{align*}
which contradicts (\ref{eqn:MB09_6}). 
\end{proof}

\begin{remark}[\textbf{Drop of the auxiliary assumption}]
Because of Theorem~\ref{thm:MA18_1}, all the results in Section~\ref{sec:scrbd} hold for unnormalized K\"ahler Ricci flows without the condition $\mathbf{scr} \geq 1$.
\label{rmk:MD28_3}
\end{remark}

\begin{remark}[\textbf{Where is the K\"ahler condition used}]
  From the proof of Theorem~\ref{thm:MA18_1}, it is clear that  lower bound of $\mathbf{scr}$ can be obtained whenever the limit space $\bar{M}$ satisfies sharper versions of the estimates in  Definition~\ref{dfn:SC02_1}.
  If the Ricci flow is on K\"ahler manifold, then the sharper estimates are observed and obtained from the compactness of the moduli $\widetilde{\mathscr{KS}}(n,\kappa)$(c.f. Theorem~\ref{thm:GC21_2}). 
  In the proof of Theorem~\ref{thm:GC21_2}, K\"ahler condition is only used to guarantee the codimension-4 condition.   Based on the work of Cheeger-Naber on the codimension-4 conjecture~\cite{CN2},  which appeared after Chen-Wang~\cite{CW6}, 
  it is almost immediate that the K\"ahler condition in Theorem~\ref{thm:GC21_2} can be dropped. 
\label{rmk:MD02_1}  
\end{remark}

\section{Proof of the main theorem}
Now we are able to finish the proof of Theorem~\ref{thm:MD09_1}. 

\begin{proof}[Proof of Theorem~\ref{thm:MD09_1}:]
By Perelman's estimate, we know that scalar curvature $R$ is uniformly bounded along the flow (\ref{eqn:MD16_3}). 
Moreover, the flow (\ref{eqn:MD16_3}) is $\kappa$-noncollapsed on the scale $1$, in the sense of Definition~\ref{dfn:MD16_1}. (c.f.~\cite{SeT}). 
Note that by parabolic scaling, for each large $t_i$, we obtain an unnormalized Ricci flow solution $\{(M, g_i(t)),  -1 \leq t <1\}$. 
Clearly, the flow $g_i$ exists on $[-0.5, 0.5]$ and has uniformly bounded scalar curvature.   
Therefore, we can choose a uniform small $r_0$ such that (\ref{eqn:MB07_2a}) holds for $T=0.5$. 
Applying Theorem~\ref{thm:MA18_1}, we see that $\mathbf{scr}_{g_i}(M, 0) \geq \epsilon r_0$ uniformly. 
Then it follows from Proposition~\ref{prn:MD16_1} that $(M, g_i(0))$ converges to $(\hat{M}, \hat{g})$ in the Cheeger-Gromov topology, for some
$\hat{M}$ with the regular-singular decomposition $\hat{M}=\mathcal{R} \cup \mathcal{S}$.  Moreover, $\dim_{\mathcal{H}} \mathcal{S} \leq 2n-4$ by Corollary~\ref{cly:MD24_1}.  
The regular part is a K\"ahler Ricci soliton since the $\mu$-functional minimizer $f_i$ of each $(M, g_i(0))$ converges to a limit function $\hat{f}$
on $\mathcal{R}(\hat{M})$ satisfying (\ref{eqn:MD09_1}).  More details can be found in Theorem 4.4 of the first paper of Chen-Wang~\cite{CW3}, whose proof
applies directly here. 
Since $(M, g_i(0))$ is isometric to $(M, g(t_i))$, we have completed the proof of Theorem~\ref{thm:MD09_1}. 
\end{proof}

\vspace{0.5in}

Xiuxiong Chen, Department of Mathematics, Stony Brook University,
NY 11794, USA;
School of Mathematics, University of Science and Technology of China, Hefei, Anhui, 230026, PR China;
xiu@math.sunysb.edu.\\

Bing  Wang, Department of Mathematics, University of Wisconsin-Madison,
Madison, WI 53706, USA;  bwang@math.wisc.edu.\\

\end{document}